\documentclass[12pt]{amsart}

\usepackage{microtype} %overfull boxes

\usepackage[left=3cm,right=3cm,bottom=2cm,top=3cm]{geometry}
\usepackage{amsmath}
\usepackage{amssymb}
\usepackage{mathrsfs}
\usepackage[all]{xy}
\usepackage[dvips]{graphicx}
\usepackage{enumerate}
\usepackage{pstricks,pst-plot,pst-node}

\theoremstyle{plain}
\newtheorem{thm}{Theorem}[section]
\newtheorem{cor}[thm]{Corollary}
\newtheorem{lemma}[thm]{Lemma}
\newtheorem{prop}[thm]{Proposition}
\newtheorem{remark}[thm]{Remark}
\theoremstyle{definition}
\newtheorem{definition}[thm]{Definition}%[section]
\newtheorem{ex}[thm]{Example}%[section]
\theoremstyle{conjecture}
\def\A{\mathcal{A}}
\def\B{\mathcal{B}}
\def\C{\mathcal{C}}
\def\D{\mathcal{D}}
\def\E{\mathcal{E}}
\def\G{\mathcal{G}}
\def\U{\mathcal{U}}
\def\I{\mathcal{I}}

\def\P{\mathcal{P}}

\def\S{\mathcal{S}}
\def\T{\mathcal{T}}

\def\ZZZ{\mathbb{Z}}

\def\Hom{\operatorname{Hom}}

\def\Ext{\operatorname{Ext}}

\def\ind{\operatorname{ind}}
\def\mod{\operatorname{mod}}

\def\obj{\operatorname{obj}}

\begin{document}

\title{Hom-configurations and Noncrossing Partitions}

\author{Raquel Coelho Sim\~oes}
\address{School of Mathematics, \\
University of Leeds. \\ Leeds LS2 9JT, UK.}
\email{rcsimoes@maths.leeds.ac.uk}

%\author{Robert Marsh}
%\address{School of Mathematics, \\
%University of Leeds. \\ Leeds LS2 9JT, UK.}
%\email{r.j.marsh@maths.leeds.ac.uk}

\date{5 December 2010}

%If omitted, automatically gives date of latex
%compilation. Useful except that good to fix date if distributing paper.

\begin{abstract}
We study maximal Hom-free sets in the $\tau [2]$-orbit category $C(Q)$ of the bounded derived category for the path algebra associated to a Dynkin quiver $Q$, where $\tau$ denotes the Auslander-Reiten translation and $[2]$ denotes the square of the shift functor. We prove that these sets are in bijection with periodic combinatorial configurations, as introduced by Riedtmann, certain $\Hom_{\leq 0}$-configurations, studied by Buan, Reiten and Thomas, and noncrossing partitions of the Coxeter group associated to $Q$ which are not contained in any proper standard parabolic subgroup. Note that Reading has proved that these noncrossing partitions are in bijection with positive clusters in the associated cluster algebra. Finally, we give a definition of mutation of maximal Hom-free sets in $\C(Q)$ and prove that the graph of these mutations is connected.
\end{abstract}


%Note this is now 2010 version but latex is too slow to allow this option!

%%%%%%%%%%%%%%%End of title information
%%%%%%%%%%%%%%%Start of article

\maketitle

\bibliographystyle{plain}

\section{Introduction}

Let $Q$ be a Dynkin quiver, $\D^b(Q)$ the bounded derived category for the path algebra associated to $Q$, with shift functor $[1]$ and Auslander-Reiten translation $\tau$. 

Our main object of study is the set of Hom-configurations in the orbit category $\C(Q) := \D^b (Q) /\tau [2]$, which is triangulated by Keller \cite{Keller}. A Hom-configuration is defined to be a maximal Hom-free set of indecomposable objects in this category. We will give bijections between the collection of Hom-configurations in $\C(Q)$ and collections of other representation-theoretic and combinatorial objects.

One such collection is the set of combinatorial configurations, which were introduced by Riedtmann \cite{Riedtmann} in order to classify self-injective algebras of finite representation type. A combinatorial configuration can be regarded as a certain Hom-free collection of indecomposable objects in the bounded derived category of finitely generated modules over a path algebra associated to a Dynkin quiver. Riedtmann proved (cf. \cite{Riedtmann, Riedtmann2}) that combinatorial configurations in type $A$ and $D$ are invariant under the functor $\tau [2]$.

Motivated by this, the authors of \cite{BMRRT} studied the Ext-version of these configurations in the bounded derived category $\D^b$ for any finite dimensional hereditary algebra, the so called Ext-configurations. The authors proved that these objects are invariant under the functor $\tau^{-1} [1]$. This implies that Ext-configurations in the bounded derived category $\D^b$ are in one-to-one correspondence with Ext-configurations in the cluster category $\D^b/ \tau^{-1} [1]$, which are called cluster-tilting objects. These objects were proved to be in bijection with maximal Ext-free sets in the cluster category (cf. \cite[2.3]{BMRRT}). One should then expect that a similar result holds for combinatorial configurations in the orbit category $\D^b(Q) / \tau [2]$, where $\D^b(Q)$ is the bounded derived category for the path algebra associated to any Dynkin quiver $Q$. 

We prove that Hom-configurations in $\C(Q)$ are in bijection with periodic combinatorial configurations, i.e., combinatorial configurations in $\D^b(Q)$ which are invariant under $\tau [2]$.\\

Riedtmann also gave a natural bijection between the set of combinatorial configurations for type $A_n$ and the set $NC (n)$ of classical noncrossing partitions of the set $\{1, \ldots, n\}$, which was introduced by Kreweras \cite{Kr} in 1972. 

Later, in the early 2000's, Brady \cite{BW2} and Bessis \cite{Bessis} independently introduced an algebraic generalization of classical noncrossing partitions. To each finite Coxeter group $W$ these authors associate a poset, called the poset of noncrossing partitions of $W$, which we denote by $NC(W)$. The posets $NC (n)$ and $NC (A_{n-1})$ are known to be isomorphic \cite{Biane}. 

The initial motivation for this article was to generalize Riedtmann's bijection to any Dynkin case, using Hom-configurations. However, a simple computation in type $D_4$ shows that the number of Hom-configurations in $\C(Q)$ is different from the number of noncrossing partitions of type $D_3$. Given this fact, it is natural to consider a special subset of the set of noncrossing partitions instead. This subset consists of the noncrossing partitions which are not contained in any proper standard parabolic subgroup. Reading \cite{Reading} has proved that this subset is in bijection with positive clusters in the associated cluster algebra (\cite{CA1}, see also \cite{FZ}). For this reason, we say that these noncrossing partitions are positive. We denote by $NC^+ (W)$ the set of positive noncrossing partitions of $W$. 

We give a combinatorial description of positive noncrossing partitions for type $A$, which allows to conclude that there 
is a one-to-one correspondence between $NC (A_{n-1})$ and $NC^+ (A_n)$. Let $f$ denote this one-to-one correspondence.

We also give a bijection $\varphi$ between positive noncrossing partitions of the Coxeter group $W_Q$ associated to $Q$ and Hom-configurations in $\C(Q)$. This bijection generalizes Riedtmann's bijection, in the sense that Riedtmann's bijection is given by the composition of $\varphi$ with $f$.\\

Our work is closely related to \cite{BRT}. In this paper the authors give a natural bijection between $m-$clusters and $m-$noncrossing partitions, for $m \geq 1$, using special classes of exceptional sequences in the bounded derived category. The sets of elements of one of these special classes of exceptional sequences are called $m-\Hom_{\leq 0}-$configurations. These configurations are contained in $\D^{\geq 0}_{\leq m}$, the full additive subcategory of the bounded derived category generated by the indecomposable objects of $KQ-\mod [i]$, with $0 \leq i \leq m$.

Hom-configurations turn out to be in bijection with $1-\Hom_{\leq 0}-$ configurations contained in $D^{(\geq 0)-}_{\leq 1}$, the full additive subcategory of the bounded derived category generated by the indecomposable objects of $KQ-\mod \cup KQ-\mod [1]$ other than the projective modules.\\

The main results presented in this paper can be summarized in the following theorem.

\begin{thm}\label{aim}
Let $Q$ be a Dynkin quiver and $\C(Q)$ the orbit category $\D^b(Q) / \tau [2]$. Then the following collections of objects are in bijection:
\begin{enumerate}
\item Hom-configurations in $\C(Q)$;
\item $\Hom_{\leq 0}$-configurations contained in $D^{(\geq 0)-}_{\leq 1}$;
\item periodic combinatorial configurations;
\item sincere Hom-free sets in $KQ-\mod$;
\item positive noncrossing partitions of the Coxeter group $W_Q$ associated to $Q$.
\end{enumerate}
\end{thm}

The paper is organized as follows. We first prove that the subcategory of $\C(Q)$
\[
\,^\perp M^\perp_Q = \{ X \in \C (Q) \, \mid \, \Hom_{\C (Q)} (X,M) = 0 = \Hom_{\C (Q)} (M,X) \},
\]
where $M$ is an indecomposable object of $\C(Q)$, is equivalent to $\C (Q^\prime)$ where $Q^\prime$ is a disjoint union of quivers of Dynkin type, whose number of vertices is $n-1$.

This result gives us a strategy to prove some of the results in this paper. This strategy is to use induction and reduce to the simpler case when we have a Hom-configuration which contains a simple projective indecomposable module. 

In Section \ref{secBRT}, we prove the bijection between $(1)$ and $(2)$ in \ref{aim}. The one-to-one correspondence between $(1)$ and $(4)$ is proved in Section \ref{secsincere}. This result is crucial to prove the relation between Hom-configurations and positive noncrossing partitions.

In Section \ref{secpositive} we give a bijection between $(1)$ and $(5)$. We note that Buan, Reiten and Thomas \cite[Theorem 7]{BRT1} provide a different Coxeter-theoretic description for the noncrossing partitions in bijection with $Hom_{\leq 0}-$ configurations contained in $\D^{(\geq 0)-}_{\leq 1}$.  

In Section \ref{secriedtmann} we give a combinatorial description of this class of noncrossing partitions for type $A$ and check that the bijection between $(1)$ and $(5)$ generalizes the bijection given by Riedtmann in type $A$.

In Section \ref{secriedtmannconf}, we prove the bijection between $(1)$ and $(3)$. We use some results in \cite{BLR} and the fact that the number of Hom-configurations is given by the so called positive Fuss-Catalan number corresponding to the Coxeter group $W_Q$ associated to the Dynkin quiver $Q$. This fact follows immediately from the bijection between Hom-configurations in $\C(Q)$ and positive noncrossing partitions. 

Finally, in Section \ref{secmutations} we give a definition of mutation of Hom-configurations and prove that the graph of these mutations is connected.

\section{Perpendicular category for Hom-configurations - Main tool}\label{sec2}

Firstly let us fix some notation. Denote by $K$ an algebraically closed field, $Q$ a simply laced Dynkin quiver with $n$ vertices, $h$ its associated Coxeter number and $KQ$ the path algebra. All modules considered will be finite-dimensional. The support of a module $M$, which we will denote by $supp (M)$, is the set of vertices $i$ of $Q$ for which $M_i \ne 0$. The bounded derived category of $KQ$-modules will be denoted by $\mathcal{D} (Q)$. We know that the indecomposable objects in $\D^b (Q)$ are of the form $M[i]$, for some indecomposable $KQ$-module $M$ and some integer $i$. If $X = M [i]$ is an indecomposable object in $D (Q)$, we denote by $\overline{X}$ the corresponding indecomposable $KQ$-module $M$, and by $d(X)$ its degree, i.e., $d(X) = d(M [i]) = i$.

We define a partial order in the set $\ind \, \D^b (Q)$, the subcategory of isomorphism classes of indecomposable objects in $\D^b (Q)$, as follows. Given $X, Y \in  \ind \, \D^b (Q)$, we say that $X \preceq Y$ if there is a path from $X$ to $Y$ in the Auslander-Reiten quiver of $\D^b (Q)$. It is clear that $\preceq$ is indeed a partial order. Fix a refinement $\leq$ of $\preceq$ to a total order.

Let $\C (Q)$ be the category $\D^b (Q) / \tau^{h-1}$, where $\tau$ is the Auslander-Reiten translate. Note that $\C (Q)$ can also be defined to be the category $\D^b (Q) / \tau [2]$, where $[ \, - \, ]$ is the shift functor and $\tau$ is the AR-translate in $\D^b (Q)$ (cf. \cite{Gabriel}).

It is easy to check that the objects in the subcategory 
\[
\E (Q) = \ind ( \mod \, A \cup (\mod A \setminus \I) [1])
\]
of $\D^b(Q)$ where $\I$ denotes the set of injective modules, is a fundamental domain for the action of $\tau [2]$ on $\ind \, \D^b (Q)$. From now on, we identify the objects in $\ind \, \C(Q)$ with their representatives in this fundamental domain.

\begin{prop}\label{homC(Q)}
Let $X, Y \in \E (Q)$.
\begin{enumerate}
\item We have $\Hom_{\D^b (Q)} (X, \tau^k Y [2k]) = 0$ for all $k \ne 0, 1$.
\item $\Hom_{\D^b (Q)} (X, \tau^k Y [2k]) \ne 0$ for at most one value of $k$.
\end{enumerate}
\end{prop}
\begin{proof}
(a) Let $k = -1$ and suppose $d(Y) = 1$. Note that $d (\tau^{-1} Y [-2]) = 0$ if and only if $Y = I [1]$ for some injective $KQ$-module $I$, which contradicts the hypothesis that $Y \in \E (Q)$. Hence we have $d (\tau^{-1} Y [-2]) = -1$ and it is clear that $\Hom_{\D^b (Q)} (X, \tau^k Y [2k]) = 0$. Finally, if $d (Y) = 0$ then $d( \tau^{-1} Y [-2]) = -2$ or $-1$, and so it is obvious that there is no map from $X$ to $\tau^{-1} Y [-2]$. The case when $k < -1$ is trivial since $d(\tau^k Y [2k])$ is negative. If $k > 2$, then $d(\tau^k Y [2k]) \geq k \geq 3$, since $d(Y) = 0$ or $1$, and so $\Hom_{\D^b (Q)} (X, \tau^k Y [2k]) = 0$. Let now $k = 2$. If $d (Y) = 1$ then $d(\tau^2 Y[4]) \in \{3, 4, 5\}$ and our claim holds. Suppose then that $d (Y) = 0$. Then $d (\tau^2 Y [4]) \in \{2, 3, 4\}$. The only nontrivial case is when $d(\tau^2 Y[4]) = 2$. But this holds if and only if $Y$ is projective and $\overline{\tau Y}$ is an injective-projective module, i.e., $\overline{\tau Y} = P_a$ where $a$ must be a source of $Q$. Therefore $\tau^2 Y [4] = I_a [2]$, and $\Hom_{\D^b (Q)} (X, I_a [2]) \ne 0$ implies that $d(X) = 2$, which is a contradiction since $X \in \E(Q)$. 

(b) Note that $\Hom_{\D^b (Q)} (X, \tau Y [2]) \simeq \Hom_{\D^b (Q)} (Y[1], X)$. Suppose $\Hom_{\D^b (Q)} (X, Y) \ne 0$. Then, in particular $X \preceq Y$, and so, by transitivity we have $X \preceq Y[1]$. Because $X \ne Y[1]$, we have by antisymmetry that $Y[1] \not\preceq X$, which implies that $\Hom_{\D^b (Q)} (Y[1], X) = 0$.
\end{proof}

\begin{prop}\label{mainprop}
Let $M$ be an indecomposable object in $\C (Q)$. The full subcategory $\,^\perp M^\perp_Q$ of $\C (Q)$ whose set of objects is
\[
\,^\perp M^\perp_Q = \{ X \in \C (Q) \, \mid \, \Hom_{\C (Q)} (X,M) = 0 = \Hom_{\C (Q)} (M,X) \}
\]
is equivalent to $\C (Q^\prime)$ where $Q^\prime$ is a disjoint union of quivers of Dynkin type, whose number of vertices is $n-1$.
\end{prop}
 
Firstly, we will show this proposition in the case when $M$ is an indecomposable simple projective $KQ$-module.

\begin{lemma}\label{mainlemma}
Let $M = P_a$ ($a \in Q_0$) be an indecomposable simple projective $KQ$-module. Then $\,^\perp M^\perp_Q$ is equivalent to $\C (Q^\prime)$ where $Q^\prime$ is the full subquiver of $Q$ whose set of vertices is $Q_0 \setminus \{a\}$.
\end{lemma}

\begin{proof}
Let $M = P_a$ be an indecomposable simple projective $KQ$-module. Recall that an indecomposable object $X$ in $\D^b(Q)$ is of the form $\overline{X} [i]$, where $i \in \ZZZ$ and $\overline{X}$ is an indecomposable module. It is easy to check that the indecomposable objects of $\,^\perp M^\perp$ are of the form:
\begin{equation}\label{Mperp}
\begin{split}
\ind \,^\perp M^\perp & = \{ X \in \ind \, KQ-\mod \, \mid \, (\underline{dim} \, X)_a = 0 \} \\
                      & \quad \cup \{\overline{X} [1] \, \mid \, \overline{X} \in \ind \, KQ-\mod, \overline{X} \, \text{noninjective}, (\underline{dim} \, X)_a = 0 \}.
\end{split}             
\end{equation}

Let $Q^\prime$ be the full subquiver of $Q$ whose set of vertices is $Q_0 \setminus \{a\}$. Let $\S_a$ denote the full subcategory of $KQ$-mod whose set of objects are the $KQ$-modules with no support at $a$, and let $\D^b_{\S_a}$ be the full subcategory of $\D^b (Q)$ whose objects are
\begin{equation*}
\obj \D^b_{\S_a} = \{ X \in \D^b (Q) \mid H^n (X) \in \S_a \, \forall n \}. \\
               %\simeq \{ X \in \D^b (Q) \mid X^n \in \S_a \forall n \}
\end{equation*}
Given that $X \simeq \oplus_{n \in \ZZZ} H^n (X) [n]$ in the hereditary case (see e.g. \cite[Lemma 3.3]{Kristian}), we have that $\D^b_{\S_a}$ is equivalent to the full subcategory whose collection of objects is $\{ X \in \D^b (Q) \mid X^n \in \S_a \, \forall n \}$.

It is easy to check that $\D^b_{\S_a}$ is triangle equivalent to $\D^b (Q^\prime)$. We denote by $G$ this triangle equivalence and we will use it to define a $K$-linear functor $F_M$ from $\,^\perp M^\perp$ to $\C(Q^\prime)$. 

In order to define $F_M$ on the objects, note that $\obj \, \,^\perp M^\perp \subseteq \obj \, \D^b_{\S_a}$, by \eqref{Mperp}. Hence, given $X \in \,^\perp M^\perp$, we can define $F_M(X)$ to be $G(X)$ regarded as an element of $\C(Q^\prime)$. Note also that if $X \in KQ-mod \cap \,^\perp M^\perp = \S_a$, then $F_M(X) \in KQ^\prime-mod$. Moreover, it follows from \cite[III.2.6 (b)]{ASS} that $X$ is an injective $KQ$-module if and only if $G(X)$ is an injective $KQ^\prime$-module, because $a$ is a sink. Hence $F_M$ maps the indecomposable objects of $\,^\perp M^\perp$ into the fundamental domain $\E(Q^\prime)$ (recall that we regard the objects of $\C (Q)$ as objects in the fundamental domain $\E (Q)$). 

It is enough to define $F_M$ on the morphisms between indecomposable objects. So let $X, Y \in \ind \,^\perp M^\perp$ and $f \in \Hom_{\C (Q)} (X,Y)$. By \ref{homC(Q)}, we have that $f \in \Hom_{\D^b (Q)} (X,Y)$ or $f \in \Hom_{\D^b (Q)} (X, \tau Y [2])$. If $f \in \Hom_{\D^b(Q)} (X,Y)$, set $F_M(f) := G(f)$. If $f \in \Hom_{\D^b(Q)} (X, \tau Y [2])$, we must have $X = \overline{X} [1]$ and $Y = \overline{Y}$. Let $\psi_Y$ denote the isomorphism $$\xymatrix{\Hom_{\D^b (Q)} (\overline{X} [1], \tau \overline{Y} [2]) \ar[r]^--{\psi_Y} & \Hom_{KQ} (\overline{Y}, \overline{X})}$$ and by $\psi^\prime_Y$ the corresponding isomorphism $$\xymatrix{\Hom_{\D^b (Q^\prime)} (\Sigma G (\overline{X}), \Sigma^2 \tau^\prime G(\overline{Y})) \ar[r]^--{\psi^\prime_Y} & Hom_{KQ^\prime} (G(\overline{Y}), G(\overline{X}))}$$ in $\D^b (Q^\prime)$, where $\Sigma$ denotes the shift functor and $\tau^\prime$ the AR-translate in $\D^b(Q^\prime)$. We then define $F_M(f)$ to be $\psi_Y^{\prime^{-1}} (G (\psi_Y (f))$, regarded as an element of $\Hom_{\C (Q^\prime)} (F_M(X), F_M(Y))$.

One can easily see that $F_M$ is indeed a functor. Since $G$ is dense, so is $F_M$, and it is easy to check that $F_M$ is also fully faithful using the definition of $F_M$ on the morphisms and the fact that $G$ is an equivalence.
\end{proof}

\begin{remark}\label{functorF}
Note that the equivalence $F_M$ given in the proof above satisfies the following properties, which are going to be useful later:
\begin{enumerate}
\item $d(F_M(X)) = d(X)$ for all $X \in \,^\perp M^\perp_Q$,
\item For each $b \in Q_0 \setminus \{a\}$, the simple $KQ^\prime$-module $S^\prime_b$ is the image of the simple $KQ$-module $S_b$,
\item For each $KQ$-module $X$ in $\,^\perp M^\perp_Q$, $supp (F_M(X)) = supp (X)$.
\end{enumerate}
\end{remark}

In order to prove \ref{mainprop} we recall the definition of section.

\begin{definition}\cite[VIII.1.2]{ASS}
Let $(\Gamma, \tau)$ be a connected translation quiver. A \textit{section} of $\Gamma$ is a connected full subquiver $\Sigma$ of $\Gamma$ satisfying the following properties:
\begin{enumerate}[i.]
\item $\Sigma$ is acyclic.
\item $\Sigma$ meets each $\tau$-orbit exactly once.
\item If $x_0 \rightarrow x_1 \rightarrow \cdots \rightarrow x_t$ is a path in $\Gamma$ with $x_0, x_t \in \Sigma_0$ then $x_i \in \Sigma_0$, for all $0 \leq i \leq t$.
\end{enumerate}
\end{definition}

We can associate a section to an arbitrary indecomposable object $M$ of $\D^b (Q)$. Indeed, let $x_0$ be the vertex of the AR-quiver $\Gamma$ of $\D^b (Q)$ associated to $M$. Recall that $\Gamma = \mathbb{Z} Q$, since $Q$ is of Dynkin type, (cf. \cite{Happel}). Let $\Sigma = \cup_k \Sigma^k$ be the full subquiver of $\Gamma$ whose set of vertices is defined as follows: 
\begin{enumerate}
\item $\Sigma^0 := \{ x_0 \}$,
\item $\Sigma^k := \{y \in \Gamma \, \mid \, x \rightarrow y \in \Gamma \text{ with } x \in \Sigma^{k-1} \text{ and } \tau \, y \not\in \Sigma^{k-1} \}$.
\end{enumerate}
Note that $M = \tau^m P_i$, for some vertex $i$ in $Q$ and some integer $m$, and so $\Sigma = \cup_{k=1}^r \Sigma^k$, where $r$ is the length of the longest unoriented path in $Q$ starting at the vertex $i$.

It is easy to prove that $\Sigma$ is in fact a section and we will call it the section associated to $M$.

The proof of \ref{mainprop} follows easily from \ref{mainlemma}.

\begin{proof}[Proof of \ref{mainprop}]
Let $\Sigma$ be the section associated to $M$ and let $\Omega$ be the quiver obtained from $\Sigma$ by reversing all the arrows. By \cite[VIII.1.6]{ASS}, we have $\mathbb{Z} Q \simeq \mathbb{Z} \Omega$, and so $\D^b (Q) \simeq \D^b (\Omega)$. Let $G$ denote this equivalence. We can assume that the image of $M$ under $G$ is the projective $K\Omega$-module associated to $x_0$. This projective module is simple since $x_0$ is a sink in $\Omega$. We can easily see that $\C (Q) \simeq \C (\Omega)$ and $\,^\perp M^\perp_Q \simeq \,^\perp G(M)^\perp_\Omega$. Because $G(M)$ is a simple projective $K \Omega$-module, it follows from \ref{mainlemma} that $\,^\perp G(M)^\perp_\Omega \simeq \C (Q^\prime)$, where $Q^\prime$ is a full subquiver of $\Omega$ with $n-1$ vertices, which finishes the proof.
\end{proof}

The main objects of our study are defined as follows.

\begin{definition}
Let $\C$ be an additive category. 
\begin{enumerate}
\item A \textit{Hom-free set} of indecomposable objects of $\C$ is a set $\T$ of indecomposable pairwise non-isomorphic objects of $\C$ such that $\Hom_{\C} (X,Y) = 0$ for all $X, Y \in \T, X \ne Y$.
\item A maximal Hom-free set in $\C$ will be called a \textit{Hom-configuration}.
\end{enumerate}
\end{definition}

We will study Hom-configurations in the quotient category $\C (Q)$.

\begin{ex}
Given an arbitrary Dynkin quiver $Q$, the set of simple $KQ$-modules is a Hom-configuration in $\C(Q)$.
\end{ex}

\begin{lemma}\label{Hom-config-BLR}
A Hom-free set in $\C (Q)$ is a Hom-configuration if and only if it has $n$ elements.
\end{lemma}

\begin{proof}
This follows easily from \ref{mainprop} using induction on the number of vertices of $Q$.
\end{proof}

\section{Hom-configurations vs Hom$_{\leq 0}$-configurations}\label{secBRT}

In this section we will see that the main object of our study, Hom-configurations in the quotient category $\C (Q)$, are very closely related to Hom$_{\leq 0}$-configurations, a object introduced by Buan-Reiten-Thomas (cf. \cite{BRT}).

Exceptional sequences are crucial for our study, and are defined as follows.

\begin{definition}\label{def.exc.seq}
\begin{enumerate}
\item An object $X$ of an abelian or triangulated category $\C$ is said to be \textit{rigid} if $\Ext^1_{\C} (X,X) = 0$. If in addition, the object $X$ is indecomposable then it is said to be \textit{exceptional}.
\item An \textit{exceptional sequence} in $KQ-\mod$ is a sequence $\mathcal{E} = (E_1, \cdots, E_n)$ of exceptional $KQ$-modules satisfying the following property
\[
\Hom_{KQ} (E_i, E_j) = 0 = \Ext^1_{KQ} (E_i, E_j), \text{for } j > i.
\]
\item An \textit{exceptional sequence} in $\D^b (Q)$ is a sequence of exceptional objects satisfying the following property
\[
\Ext^m_{\D^b (Q)} (E_i, E_j) = 0, \text{for } j > i \text{ and } m \in \ZZZ.
\]
\end{enumerate}
\end{definition}

In order to simplify the exposition, we use the reverse of the usual convention for the order of an exceptional sequence. 

\begin{lemma}\label{differentnotionexcseq}
An exceptional sequence in $\D^b (Q)$ can also be defined to be a sequence $(X_1, \ldots, X_n)$ of indecomposable objects such that $(\overline{X}_1, \ldots, \overline{X}_n)$ is an exceptional sequence in $KQ-\mod$.
\end{lemma}
\begin{proof}
Suppose $(X_1, \ldots, X_n)$ is an exceptional sequence in $\D^b(Q)$ as defined in \ref{def.exc.seq} (3). Since for each $i \in [n]$, $X_i$ is indecomposable, we have $X_i = \overline{X}_i [t_i]$, for some integer $t_i$. In other words, $\overline{X}_i = X_i [-t_i]$, for all $i$. Note also that we have $\Ext^k_{KQ} (\overline{X}_i, \overline{X}_j) \simeq \Ext^k_{\D^b (Q)} (\overline{X}_i,\overline{X}_j)$, for $k = 0, 1$ and  $j > i$, since $\overline{X}_j, \overline{X}_i$ are $KQ$-modules. %Happel's book, pg 30
Hence, for $j > i$ and $k = 0, 1$, we have that
\[
\Ext^k_{KQ} (\overline{X}_i, \overline{X}_j) \simeq \Hom_{\D^b (Q)} (X_i [-t_i], X_j [k-t_j]) \simeq \Hom_{\D^b (Q)} (X_i, X_j [k-t_j+t_i]) = 0,
\]
by \ref{def.exc.seq}. 

Conversely, suppose $(X_1, \ldots, X_n)$ is a sequence of indecomposable objects in $\D^b (Q)$ such that $(\overline{X}_1, \ldots, \overline{X}_n)$ is an exceptional sequence in $KQ-\mod$. Then we have that $\Hom_{\D^b(Q)} (\overline{X}_i, \overline{X}_j [k]) = 0$, for every integer $k$ and $j > i$, by assumption in the case when $k = 0, 1$ and because $\overline{X}_j$ and $\overline{X}_i$ are $KQ$-modules, in the case when $k \in \ZZZ \setminus \{0,1\}$. Using the same argument as used above, we easily deduce that $(X_1, \ldots, X_n)$ is an exceptional sequence in $\D^b (Q)$ as defined in \ref{def.exc.seq}. 
\end{proof}

In \cite{BRT}, Buan, Reiten and Thomas define a new object in the bounded derived category $\D^b (Q)$ of an arbitrary hereditary Artin algebra, called a $\Hom_{\leq 0}$-configuration.

\begin{definition}\label{Hom0conf.def} \cite{BRT}
An object $X \in \D^b (Q)$ is a \textit{$\Hom_{\leq 0}$ -configuration} if it satisfies the following axioms:
\begin{enumerate}
\item $X$ has $n$ indecomposable pairwise non-isomorphic summands $X_1, \ldots, X_n$, and they are rigid.
\item $\Hom_{\D^b (Q)} (X_i, X_j) = 0$ for $i \ne j$.
\item $\Ext^k_{\D^b(Q)} (X,X) = 0$ for $k < 0$.
\item The indecomposable direct summands can be ordered into an exceptional sequence. %{\red (I haven't defined exceptional sequences in $mod \, KQ$ or in $D$ yet.)}
\end{enumerate}
\end{definition}

Our aim is to prove that $Hom_{\leq 0}$-configurations contained in $\E (Q)$ are precisely the Hom-configurations in $\C (Q)$.
The following lemmas will be useful later.

\begin{lemma}\label{negative.exts}
For any pair of objects $X, Y$ in $\D^b (Q)$, we have
\[
\Hom_{\D^b (Q)} (X, \tau Y [2]) \simeq \Ext^{-1}_{\D^b (Q)} (Y, X).
\]
\end{lemma}
\begin{proof}
By Serre duality, we have that $\Hom_{\D^b (Q)} (X, \tau Y [2]) \simeq \Ext^1_{\D^b (Q)} (Y[2], X)$. And so
\begin{equation*}
\begin{split}
\Hom_{\D^b(Q)} (X, \tau Y [2]) &\simeq \Ext^1_{\D^b(Q)} (Y[2], X) \\
                             &\simeq \Hom_{\D^b (Q)} (Y[2], X[1]) \\
                             &\simeq \Hom_{\D^b (Q)} (Y, X[-1]) \\
                             &\simeq \Ext^{-1}_{\D^b (Q)} (Y, X).
\end{split}
\end{equation*}
\end{proof}

\begin{lemma}\label{pr3.Hom0confgs}
Let $\T$ be a Hom-configuration in $\C (Q)$, and let $X, Y \in \T$. We have that $\Ext^i_{\D^b(Q)} (X, Y) = 0$, for $i \leq 0$.
\end{lemma}
\begin{proof}
Note that $X, Y \in \E$. So, we have
\[
0 = \Hom_{\C(Q)} (X,Y) \simeq \Hom_{\D^b(Q)} (X,Y) \oplus \Hom_{\D^b (Q)} (X, \tau Y[2]).
\]
Hence $\Hom_{\D^b(Q)} (X, Y) = 0$. We also have $\Hom_{\C(Q)} (Y,X) = 0$, which implies that $\Hom_{\D^b (Q)} (Y, \tau X [2]) = 0$, and so $\Ext^{-1}_{\D^b (Q)} (X, Y) = 0$, by \ref{negative.exts}. 
For $i \leq -2$, we have that $Y[i]$ has negative degree, as $Y$ lies in $\E$. Therefore $\Ext^i_{\D^b (Q)} (X, Y) = 0$, for $i \leq -2$, since $X$ has degree $0$ or $1$.
\end{proof}

\begin{remark}\label{facts}
\begin{enumerate}
\item Let $P$ be an indecomposable projective $KQ$-module. If \linebreak $\Hom_{\D^b(Q)} (P,X) \ne 0$ then $X \in KQ-\mod$.
\item Any non-zero $KQ$-module has a non-zero morphism to an indecomposable injective module.
\item Let $P$ be an indecomposable projective $KQ$-module. We have that $P \preceq I$ for any indecomposable injective module $I$.
\end{enumerate}
\end{remark}
\begin{proof}
We just prove $(3)$. Let $P$ be an indecomposable projective $KQ$-module and let $I$ be an arbitrary indecomposable injective module. Consider the section $\Sigma$ associated to $P$, as defined as in the proof of \ref{mainprop}. Then $\Sigma$ meets the $\tau$-orbit of $I$ at exactly one point $X$. Due to the definition of this section, we have $P \preceq X$ and because $X$ lies in the $\tau$-orbit of $I$, we have $X \preceq I$, and so the assertion follows by transitivity.  
\end{proof}

\begin{prop}\label{Hom-confgs-exc.seq}
Let $\T$ be a Hom-configuration in $\C (Q)$. If we order the elements of $\T$ respecting the total order $\leq$ (i.e., order the elements from left to right in the AR-quiver), we obtain an exceptional sequence in $\D^b (Q)$.
\end{prop}
\begin{proof}
Let $\T = \{X_1, \ldots, X_n\}$ be a Hom-configuration of $\C (Q)$ ordered with respect to the total order $\leq$. We want to check that 
\begin{equation}
\Ext^m_{\D^b (Q)} (X_i, X_j) = 0, \label{exts}
\end{equation}
for $j > i$ and for any integer $m$.

We have that \eqref{exts} holds for $m \leq 0$, by \ref{pr3.Hom0confgs}. \eqref{exts} also holds for $m > 2$ since $X_i$ has degree $0$ or $1$ and $X_j [m]$ has degree $\geq 3$. 

Let us check the case when $m = 1$. Given $i, j \in [n]$ with $i < j$, we have $\Hom_{\D^b(Q)} (X_i, X_j)$ \linebreak $\simeq \Hom_{\D^b (Q)} (X_j, \tau X_i)$ by Serre duality. If this is non-zero, then we have in particular that there is a path from $X_j$ to $\tau X_i$ in the AR-quiver of $\D^b (Q)$, which implies that $X_j \leq \tau X_i$. On the other hand, $\tau X_i \leq X_i$, and so by transitivity, we have $X_j \leq X_i$. Since $j > i$, we have $X_i \leq X_j$, so $X_i = X_j$ by antisymmetry, a contradiction. Hence, \eqref{exts} holds for $m = 1$.

Finally, let us check for $m = 2$. Note that $X_i \leq X_j$, and so $X_j \not\preceq X_i$, i.e., there is no path from $X_j$ to $X_i$ in the AR-quiver of $\D^b (Q)$. 

We have that $X_i = \tau^{-l} (P_a)$, for some natural number $l$ and some indecomposable projective $KQ$-module $P_a$ ($a$ denotes the vertex corresponding to the projective module).

Suppose $\Ext^2_{\D^b (Q)} (X_i, X_j) = \Hom_{\D^b(Q)} (P_a, \tau^l X_j [2]) \ne 0$. By \ref{facts} (1), $\tau^l X_j [2]$ is a $KQ$-module, i.e., $\tau^l X_j$ has degree $-2$. It follows from \ref{facts} (2) that there is a path from $\tau^l X_j$ to $I[-2]$, for some indecomposable injective module $I$, i.e., $\tau^l X_j \preceq I[-2]$. We also have that $I[-2] \preceq P[-1]$, where $P$ is the indecomposable projective such that $soc \, I \simeq P / rad \, P$, in other words, $\tau P = I [-1]$. By \ref{facts} (3), we have $P [-1] \preceq I_a[-1]$. We also have $I_a [-1] \preceq P_a$. By transitivity we can conclude that $\tau^l X_j \preceq P_a$, and so $X_j = \tau^{-l} (\tau^l X_j) \preceq \tau^{-l} P_a = X_i$, which is a contradiction.
\end{proof}

\begin{thm}\label{1-1corHom<0}
Hom-configurations in $\C (Q)$ are in $1-1$ correspondence with $\Hom_{\leq 0}$-configurations contained in $\E (Q)$.
\end{thm}
\begin{proof}
Let $\T$ be a Hom-configuration of $\C (Q)$. By \ref{Hom-config-BLR} $\T$ has $n$ elements, and every element of $\T$ is rigid, since $Q$ is of Dynkin type. Properties $2$ and $3$ of \ref{Hom0conf.def} follow from \ref{pr3.Hom0confgs} and the definition of Hom-configuration and property $4$ follows from \ref{Hom-confgs-exc.seq}. Conversely, suppose $\T$ is a $\Hom_{\leq 0}$-configuration of $\D^b (Q)$ contained in $\E$. Then by property $1$, it has $n$ elements, and for any pair of objects $X,Y$ of $\T$ we have
\[
\Hom_{\C(Q)} (X,Y) = \Hom_{\D^b(Q)} (X,Y) \oplus \Hom_{\D^b(Q)} (X, \tau Y [2]),
\]
since $X, Y$ lie in $\E$. Both summands are zero, by properties $2, 3$ and by \ref{negative.exts}. Hence $\T$ is a Hom-free set with $n$ elements, and so the result follows from \ref{Hom-config-BLR}. 
\end{proof}

Note that the full subcategory $\D^{(\geq 0)-}_{\leq 1}$ of $\D^b (Q)$ whose indecomposables are in $KQ -\mod [i] \setminus \{P_i \, \mid \, i \in [n] \}$ with $i = 0, 1$, which was considered in \cite{BRT}, is just a different fundamental domain for the action of $\tau [2]$ in $\D^b (Q)$. Hence Hom-configurations in $\C(Q)$ are in 1-1 correspondence with $\Hom_{\leq 0}$-configurations contained in $\D^{(\geq 0)-}_{\leq 1}$, via the auto-equivalence $\tau^{-1}$.

\section{Sincere Hom-free sets in $KQ-\mod$}\label{secsincere}

This section is devoted to the study of the set of modules of a Hom-configuration. We prove that the restriction of the Hom-configurations in $\C(Q)$ to $KQ-\mod$ is precisely the set of sincere Hom-free sets.

Some of the proofs  will rely on using reflection functors, which correspond to changing the orientation in the quiver $Q$, to reduce to the case when we have a simple projective module, so we can use \ref{mainlemma}.

Given a sink or a source $i$ of the quiver $Q$, we denote by $\sigma_i (Q)$ the quiver obtained from $Q$ by reversing all the arrows incident to $i$. We denote by $R_i$ the \textit{(simple) reflection functor} associated to a sink $i$ and by $R_j^-$ the simple reflection functor associated to a source $j$. If $i$ is a sink of $Q$, the functor $R_i$ gives an equivalence between $\D^b (Q)$ and $\D^b (\sigma_i (Q))$, and the inverse is given by $R_i^-$. Because $R_i$ and $R_i^-$ commute with $\tau$ and $[ \, - \, ]$, these functors induce equivalences between $\C(Q)$ and $\C(\sigma_i (Q))$. If we have a sequence $i_1, \ldots, i_k$ of vertices of $Q$ such that each $i_j$ is a sink in $\sigma_{i_{j-1}} \ldots \sigma_{i_1} (Q)$, and $R$ is the sequence of reflections $R_{i_k} \ldots R_{i_1}$, we denote by $\sigma_R (Q)$ the quiver $\sigma_{i_k} \ldots \sigma_{i_1} (Q)$, for simplicity. 

We have the following useful description of the image of an indecomposable object of $\D^b(Q)$ under these reflection functors:

Let $i$ be a sink (source) of $Q$, and $M [j]$ an indecomposable object of $\D^b(Q)$. If $M \ne S_i$, then $R_i (M[j]) = N [j]$ ($R^-_i (M [j]) = N [j]$), where $N$ is the indecomposable $K\sigma_i (Q)$-module whose dimension vector is $\underline{dim} (N) = s_i (\underline{dim} M)$, where $s_i$ is the simple reflection associated to the simple root $\alpha_i$. If $M = S_i$, then $R_i (M [j]) = M [j-1]$ ($R^-_i (M [j]) = M [j+1]$).

\begin{remark}\label{reflectionfunctor}
\begin{enumerate}
\item Given $X \in \C (Q)$, there exists a composition of reflection functors $R_{i_k} \ldots R_{i_1}$, where $i_j$ is a sink in $\sigma_{i_{j-1}} \ldots \sigma_{i_1} (Q)$ for each $2 \leq j \leq k$ such that $R_{i_k} \ldots R_{i_1} (X)$ is a simple projective $K (\sigma_{i_k} \ldots \sigma_{i_1} (Q))$-module.
\item Given a set $\T$ of $n$ objects in $\C (Q)$ and a sink $i$ in $Q$, we have that $\T$ is a Hom-configuration in $\C(Q)$ if and only if $R_i (\T)$ is a Hom-configuration in $\C( \sigma_i (Q))$.
\end{enumerate}
\end{remark}

\begin{proof}
Part (1) is a well known result, but we will give a specific sequence of reflection functors which will be useful later.

Consider the set of objects $\{ Y \in \E(Q) \, \mid \, Y \preceq X \}$. Choose a refinement of $\preceq$ to a total order in this set and write the elements in order with respect to this refinement. Let $\{ Y_1, \ldots, Y_k \}$ be this ordering. Given that $Y_1$ is $\preceq-$minimal, $Y_1$ must be a simple projective $KQ$-module. Let $i_1$ be the sink of $Q$ associated to $Y_1$, i.e., $Y_1 = P_{i_1}$. Note that $Y_{2}$ is $\preceq-$minimal in $\E(\sigma_{i_1} (Q))$, i.e., $Y_{2}$ is a simple projective $K \sigma_{i_1} (Q)$-module. Let $i_2$ be the corresponding sink in $\sigma_{i_1} (Q)$. Proceeding this way, we get the composition $R_{i_k} \ldots R_{i_{2}} R_{i_1}$ of reflection functors. Clearly this composition maps $X$ to a simple-projective $K \sigma_R (Q)$-module.
 
Part (2) follows easily from the fact that $R_i$ is an equivalence and it commutes with $\tau$ and the shift functor. 
\end{proof}

\begin{lemma}\label{Hom-confgsandmodules}
\begin{enumerate}
\item \cite[Theorem 3]{Ringel} The set of the simple $KQ$-modules is the unique Hom-configuration of $\C(Q)$ consisting of modules.
\item Any Hom-configuration of $\C(Q)$ has at least one $KQ$-module.
\end{enumerate}
\end{lemma}

\begin{proof}
Part (1) was proved by Ringel (cf.\cite[Theorem 3]{Ringel}) but we will give an alternative proof, which will be by induction on $n$, the number of vertices of $Q$. The case when $n= 1$ is very easy to check. Let $\T$ be a Hom-configuration consisting of modules in $\C (Q)$, where $|Q_0| = n$. First suppose that $\T$ contains a simple projective module $S_i$. We will use the equivalence $F_{S_i}$ between $\,^\perp S_i^\perp$ and $\C (Q^\prime)$, where $Q^\prime = Q \setminus \{i\}$ defined in Section \ref{sec2} (for its definition and some of its properties, see proof of \ref{mainlemma} and \ref{functorF}). We have that $F_{S_i}(\T \setminus S_i)$ is a Hom-configuration in $\C (Q^\prime)$ and it consists only of $KQ^\prime$-modules, by \ref{functorF} (1). It follows by induction that these $KQ^\prime$-modules are the simple $KQ^\prime$-modules, and so we have $\T = \{S_1, \ldots, S_n\}$, by \ref{functorF} (2). 

Suppose now that $\T$ doesn't contain any simple projective module. Let $X$ be a minimal element of $\T$ with respect to the partial order $\preceq$. Let $R$ be the composition of reflection functors described in the proof of \ref{reflectionfunctor} (1). We have that $R$ maps $X$ to a simple-projective $K \sigma_R (Q)$-module and if $Y \in KQ-\mod$ with $Y \not\preceq X$ then $R(Y) \in K \sigma_R (Q)-\mod$. The same holds for $R^\prime := R_{i_{k-1}} \ldots R_{i_2} R_{i_1}$, i.e., $R^\prime (Y) \in K \sigma_{R^\prime} (Q)-\mod$. So, in particular, $R(\T)$ lies in $K \sigma_R (Q)-\mod$, due to the choice of $X$. Moreover, by \ref{reflectionfunctor} (2), $R (\T)$ is a Hom-configuration in $\C (\sigma_R (Q))$ and it contains a simple projective module. Hence, $R(\T)$ is the set of simple modules in $K \sigma_R (Q)-\mod$. However, if one considers the simple injective $K \sigma_R (Q)$-module $S^\prime$ corresponding to $R_{i_k}$, we know there is an element $Y$ of $\T$ such that $R(Y) = S^\prime$, and $R_{i_{k-1}} \ldots R_{i_1} (Y)$ has degree $0$ but $R_{i_{k-1}} \ldots R_{i_1} (Y) = R^{-1}_{i_k} (S^\prime) = S^\prime [1]$, a contradiction.

To prove (2), suppose $\T$ is a Hom-configuration of $\C(Q)$ whose objects lie in $\ind ((\mod \, KQ \setminus \I) [1])$. Then $\T$ must be the set $\{ S [1] \mid S \text{ simple module} \}$, since otherwise $\T [-1] = \{ M[-1] \mid M \in \T\}$ would be a Hom-configuration consisting of modules which is not the set of the simple modules, contradicting (1). However, since $Q$ is a Dynkin quiver, it must have a source $i$. Hence $S_i = I_i$ is injective, a contradiction.
\end{proof}

In order to extend a sincere Hom-free set in $KQ-\mod$ to a Hom-configuration by adding indecomposable objects of degree $1$, we use the notion of perpendicular category. If $\T$ is a set of indecomposable modules, the perpendicular category is defined by
\[
\T^\perp = \{M \in KQ-\mod \mid \Hom_{KQ} (X,M) = 0, \Ext^1_{KQ} (X,M) = 0, \forall X \in \T \}.
\]

If $Q$ is a Dynkin quiver and $X$ is an indecomposable $KQ$-module, it is well known that $X^\perp_Q$ is equivalent to $KQ^\prime-\mod$ where $Q^\prime$ is a quiver with no oriented cycles and $|Q_0|-1$ vertices. Note that the functor from $KQ^\prime-\mod$ to $KQ-\mod$ is exact and induces isomorphisms on both $\Hom$ and $\Ext$. We refer the reader to \cite[Thm 2.3]{Schofield} for more details.

\begin{remark}\label{Schofieldcor}
Let $\T$ be a Hom-free set in $KQ-\mod$. Then $\T^\perp_Q \simeq KQ^\prime-\mod$ for some Dynkin quiver $Q^\prime$ with $|Q_0|-|\T|$ vertices.
\end{remark}

\begin{proof}
This follows immediately from Theorem 2.5 in \cite{Schofield} and \ref{Hom-confgs-exc.seq}.
\end{proof}

\begin{prop}\label{sincerity1}
Let $\T$ be a sincere Hom-free set in $KQ-\mod$. There exists a unique Hom-configuration of $\C(Q)$ whose restriction to $KQ-\mod$ is $\T$.
\end{prop}

\begin{proof}
Let $\T = \{X_1, \ldots, X_k\}$ be a sincere Hom-free set in $KQ-\mod$. Since $\T$ is sincere, we have that for any injective module $I$, there exists $i \in [k]$ such that $\Hom_{KQ} (X_i, I) \ne 0$. Hence the injective $KQ$-modules don't lie in $\T^\perp$.

We claim that given $Y \in \ind (KQ-\mod \setminus \I)$, $Y[1]$ lies in $\,^\perp \T^\perp$ if and only if $Y \in \T^\perp$. Indeed, given $X \in \T$, we have $\Hom_{\C (Q)} (X, Y [1]) \simeq \Hom_{\D^b(Q)} (X, Y[1]) \simeq \Ext^1_{KQ} (X,Y)$, and $\Hom_{\C(Q)} (Y[1], X) = \Hom_{\D^b(Q)} (Y[1], \tau X [2]) \simeq \Hom_{KQ} (X,Y)$, and so our claim holds. 

We know $\T^\perp \simeq KQ^\prime-\mod$, for some Dynkin quiver $Q^\prime$ with $n-k$ vertices.

Consider $U = \{ S [1] \, \mid \, S \text{ simple object in } \T^\perp \}$. By the first part of this proof, $U \subseteq  \,^\perp \T^\perp$ and so $\T^\prime \sqcup U$ is a Hom-free set in $\C(Q)$. Since $U$ has $n-k$ elements, we have that $\T^\prime \sqcup U$ is indeed a Hom-configuration in $\C(Q)$.

To prove the uniqueness let $V$ be a set of elements in $\ind (KQ-\mod \setminus \I) [1]$ such that $\T^\prime \sqcup V$ is a Hom-configuration in $\C (Q)$. Then it follows from the first part of the proof that $V [-1]$ is a Hom-free set in $\T^\perp$ and it contains $n-k$ elements. But $\T^\perp \simeq KQ^\prime-\mod$, where $Q^\prime$ has $n-k$ vertices, so it follows from \ref{Hom-confgsandmodules} (1) that $V$ must be the shift of the simple objects in $\T^\perp$.
\end{proof}

We remark that the proof gives us an explicit way to extend a sincere Hom-free set to a Hom-configuration, by simply taking the perpendicular category of the sincere Hom-free set and picking the shift of the simple objects in this category. 

\begin{ex}\label{examplenotation}
Consider the quiver $Q:$ $\xymatrix{4 \ar[r] & 3 & 2 \ar[l] \ar[r] & 1}$ of type $A_4$. We denote an indecomposable module with dimension vector given by $(i_1, i_2, i_3, i_4)$ by listing the vertices of the simple modules in its support. 

Let $\T = \{34, 12 \}$. Then $\obj \, \T^\perp = \{1, 123, 23\}$ and the simple objects of this subcategory are $1$ and $23$. So $\U = \{1 [1], 23 [1] \}$ and the unique Hom-configuration whose set of modules is $\T$ is $\T \sqcup \U = \{ 12, 34, 1[1], 23[1]\}$. Note that $\T \sqcup \U$ is not the unique Hom-configuration containing $\T$. For instance, $\{12, 34, 23, 123 [1]\}$ is also a Hom-configuration, but it contains a module not in $\T$.
\end{ex}

\begin{lemma}\label{SincereReflection}
Let $\T$ be a Hom-configuration in $\C (Q)$ such that $\T^\prime := \T \mid_{KQ-\mod}$ is sincere, and let $i$ be a source in $Q$. Then $R^-_i (\T) \mid_{K \sigma_i (Q) - \mod}$ is also sincere.
\end{lemma}

\begin{proof}
We will use the following notation for simplicity: $\U = R_i^- (\T)$, $\U^\prime = R_i^- (\T) \cap K \sigma_i (Q)-\mod$ and $\U^{\prime \prime} = \U \setminus \U^\prime$.

Note that all $KQ$- modules except $S_i$ are mapped to $K\sigma_i (Q)$- modules via $R^-_i$, and moreover, the support on all the vertices other than $i$ remains unchanged, so we only need to analyze what happens to the support on vertex $i$.

Suppose $S_i = I_i$ belongs to $\T$. Let $j$ be a neighbor of $i$ in $Q$. There is a $KQ$-module $Y$ in $\T^\prime$ with $(\underline{dim} \, Y)_j \ne 0$, by hypothesis. On the other hand, we must have $(\underline{dim} \, Y)_i = 0$, otherwise $\Hom (Y, I_i) \ne 0$, which contradicts the fact that $\T$ is Hom-free. Therefore $(\underline{dim} \, R^-_i (Y))_i \ne 0$ and the sincerity is preserved, as we wanted.

Suppose now that $I_i$ doesn't lie in $\T$. Assume, for a contradiction, that $\U^\prime$ is not sincere, i.e., no $K \sigma_i (Q)$-module  in $R^-_i (\T)$ has support on $i$. Let $I^i$ be the indecomposable injective $K \sigma_i (Q)$-module associated to the vertex $i$. Then we have $\Hom_{K\sigma_i (Q)} (X, I^i) = 0$, for all $X \in \U^\prime$. We also have $\Ext^1_{K \sigma_i (Q)} (X, I^i) = \Hom (X, I^i [1]) = 0$, since $I^i$ is injective and $d(X) = 0$. Hence $I^i \in \U^{\prime^\perp}$. By \ref{Schofieldcor}, we have $\U^{\prime^\perp} \simeq^G KQ^\prime-\mod$, where $Q^\prime$ is a Dynkin quiver with $n-|\U^\prime|$ vertices. Let $G (I^i) \rightarrow S \rightarrow 0$ be a surjection in $KQ^\prime-\mod$ (note that every module maps to a simple), and $K$ its kernel. Then we have a short exact sequence $0 \rightarrow G^{-1} (K) \rightarrow I^i \rightarrow G^{-1} (S) \rightarrow 0$ in $\U^{\prime^\perp}$. Since $I^i$ is injective, so is $G^{-1} (S)$. On the other hand, we have that the image of $\U^{\prime \prime} [-1]$ under $G$ is the set of simple $KQ^\prime$-modules and $\U^{\prime \prime} [-1]$ doesn't contain any injective $K \sigma_i (Q)$-module (see proof of \ref{sincerity1}), so $G^{-1} (S)$ is not injective, a contradiction.
\end{proof}

\begin{prop}\label{sincerity2}
The set of modules of any Hom-configuration in $\C(Q)$ is sincere.
\end{prop}

\begin{proof}
We prove this by induction on $n$. The proposition is trivial in the case when $n = 1$. Now suppose the proposition holds for $n-1$ and let $\T$ be a Hom-configuration in $\C(Q)$ with $|Q_0| = n$. If $\T$ has a simple projective $KQ$-module $P_i$, then the set of modules of the Hom-configuration  $F_{P_i} (\T \setminus P_i)$ in $\C(Q^\prime)$, with $Q^\prime = Q \setminus \{i\}$, is sincere by induction. So the set of modules in $\T \setminus P_i$ has support on every vertex of $Q$ except $i$, by \ref{functorF} (3). But $P_i$ has support on $i$, so $\T \mid_{KQ-\mod}$ is sincere. Suppose now that $\T$ doesn't have any simple projective module. Let $X$ be an element of $\T$. We know that there is a sequence of reflections $R_{i_1}, \ldots, R_{i_k}$ such that $ R_{i_k} \ldots R_{i_1} (X)$ is a simple projective $K \sigma (Q)$-module, with $\sigma (Q) = \sigma_{i_k} \ldots \sigma_{i_1} (Q)$. It follows from what was proved above that $R_{i_k} \ldots R_{i_1} (\T) \mid_{K\sigma (Q) - \mod}$ is sincere, and so the proposition follows immediately from \ref{SincereReflection}.
\end{proof}

\begin{thm}\label{1-1corsincere}
Let $\beta$ be the map from the collection of sincere Hom-free sets in  $KQ-\mod$ to the collection of Hom-configurations in $\C(Q)$ defined as follows. Given a sincere Hom-free set $\T$ in $KQ-\mod$, $\beta (\T) := \T \sqcup U$, where
\[
\U = \{S[1] \, \mid \, S \text{ simple object in } \T^\perp \}.
\] 
Then this map is a bijection, and its inverse is given by the restriction to $KQ-\mod$.
\end{thm}

\begin{proof}
This follows immediately from \ref{sincerity1} and \ref{sincerity2}.
\end{proof}

\section{Positive noncrossing partitions}\label{secpositive}

Let $W$ be a finite Coxeter group, $S$ the set of simple generators of $W$ and $T$ the set of reflections. Fix a Coxeter element $c$ in $W$. For $w \in W$, let $l_T (w)$ denote the \textit{absolute length} of $w$, which is the minimum length of $w$ written as a product of reflections. Given $w \in W$, we call a minimum length expression for $w$ written as a product of reflections as \textit{$T$-reduced expression}. This length naturally induces a partial order $\leq_T$ on $W$, which will be called the \textit{absolute order}.

\begin{definition}
The \textit{absolute order} $\leq_T$ is defined as follows: $$u \leq_T v \Leftrightarrow l_T (v) = l_T (u) + l_T (u^{-1} v),$$for all $u, v \in W$.
\end{definition}

Another way to define absolute order is by saying that $u \leq_T v$ if there is a $T$-reduced expression for
$v$ in which an expression for $u$ appears as a prefix.

\begin{definition} \cite{Bessis, BW2}
A \textit{noncrossing partition} associated to $W$ is an element $w \in W$ satisfying $1 \leq_T w \leq_T c$. The poset of noncrossing partitions associated to $W$ is denoted by $NC (W)$.
\end{definition}

We will state here a lemma proved by Reading \cite{Reading} which will be useful later.

\begin{lemma}\cite[Lemma 5.2]{Reading}\label{Readinglemma}
Let $x \leq_T c$, $s$ be a simple reflection and $W_{S \setminus \{s\}}$ be the standard parabolic subgroup generated by every simple reflection but $s$. Then the following are equivalent:
\begin{enumerate}
\item $x \in W_{S \setminus \{s\}}$.
\item Every reflection $t$ in any $T$-reduced expression for $x$ lies in $W_{S \setminus \{s\}}$.
\end{enumerate}
\end{lemma}

Let now $W_Q$ be the Coxeter group associated to the simply laced Dynkin quiver $Q$, and fix a Coxeter element $c = s_{i_1} \ldots s_{i_n}$ adapted to the quiver $Q$ with respect to sinks, i.e., $i_1$ is a sink of $Q$, and $i_k$ is a sink of the quiver $\sigma_{i_{k-1}} \ldots s_{i_2} s_{i_1} (Q)$, for each $k \geq 2$. The cardinality of $NC (W_Q)$ is given by the Catalan number associated to $Q$ (cf. \cite{Bessis}), which is bigger than the number of Hom-configurations in the quotient category $\mathcal{C} (Q)$. For instance, it is very easy to check that in type $A_3$ there are only $5$ Hom-configurations whereas the number of noncrossing partitions is $14$. 

\begin{definition} \cite{Reading}
A noncrossing partition which is not contained in any proper standard parabolic subgroup is said to be a \textit{positive noncrossing partition}.
\end{definition}

It was proved in \cite{Reading} that the number of positive noncrossing partitions is given by the so called \textit{positive Fuss-Catalan number} $C^+ (W_Q)$, which is defined as follows (see \cite{FZ2}):
\[
\prod_{i=1}^n \frac{e_i + h - 1}{e_i + 1},
\] 
where $h$ is the Coxeter number of $W_Q$ and $e_1, \ldots, e_n$ its exponents. 

The following table (cf. table $4$ in \cite{FZ2}) shows the explicit formulas for Dynkin type.

\begin{center}
  \begin{tabular}{| c || c | c | c | c | c |}
    \hline
    $Q$ & $A_n$ & $D_n$ & $E_6$ & $E_7$ & $E_8$\\ \hline
    $C^+ (W_Q)$ & $\frac{1}{n+1} \binom{2n}{n}$ & $\frac{3n-4}{n} \binom{2n-3}{n-1}$ & $418$ & $2431$ & $17342$ \\
    \hline
  \end{tabular}
\end{center}

In order to prove that there is a bijection between positive noncrossing partitions and Hom-configurations in $\C(Q)$ we will need to use the braid group action on the set of exceptional sequences of a fixed length. This action can also be called mutations of exceptional sequences. 

We will now recall the notion of this braid group action and some useful facts. For more details we refer the reader to \cite{Bill}.

Given an exceptional sequence $E$ in $KQ-\mod$, let $C(E)$ denote the smallest full subcategory of $KQ-\mod$ which contains $E$ and is closed under extensions, kernels and cokernels.

 Let $B_r$ be the braid group on $r$ strings, with generators $\sigma_1, \ldots, \sigma_{r-1}$ satisfying the braid relations $\sigma_i \sigma_j  = \sigma_j \sigma_i$ if $|i-j| \geq 2$ and $\sigma_i \sigma_{i+1} \sigma_i = \sigma_{i+1} \sigma_i \sigma_{i+1}$.

\begin{prop} \cite{Bill} \label{mutationsexcseqproperties}
The following holds:
\begin{enumerate}
\item Given an exceptional sequence $E$ in $KQ-\mod$, $C(E)$ is equivalent to $KQ^\prime-\mod$ where $Q^\prime$ is a quiver with no oriented cycles and with the number of vertices given by the length of $E$.
\item If $(X,Y)$ is an exceptional sequence in $KQ-\mod$ then there are unique indecomposable modules $R_Y \, X, L_X \, Y$ such that $ (Y, R_Y \, X), (L_X \, Y, X)$ are exceptional sequences in $C(X,Y)$. 
\item Let $E = (X_1, \ldots X_r)$ be an exceptional sequence and $1 \leq i < r$. Then\linebreak $(X_1, \ldots X_{i-1}, X_{i+1}, Y, X_{i+2}, \ldots, X_r)$ is an exceptional sequence in $C(E)$ if and only if $Y \simeq R_{X_{i+1}} \, X_i$. \\
Analogously,  $ (X_1, \ldots X_{i-1}, Z, X_{i}, X_{i+2}, \ldots, X_r)$ is an exceptional sequence in $C(E)$ if and only if $Z \simeq L_{X_{i}} \, X_{i+1}$.
\item The braid group $B_r$ acts on the set of exceptional sequences of length $r$ by
\[
\sigma_i (X_1, \ldots, X_r) = (X_1, \ldots X_{i-1}, X_{i+1}, R_{X_{i+1}} \, X_i, X_{i+2}, \ldots, X_r),
\]
\[
\sigma_i^{-1} (X_1, \ldots, X_r) = (X_1, \ldots X_{i-1},L_{X_{i}} \, X_{i+1} , X_{i}, X_{i+2}, \ldots, X_r).
\]
\item The braid group action preserves the product of the corresponding reflections in the Weyl group. 
\item The braid group $B_n$ on $n$ strings acts transitively on the set of complete exceptional sequences.
\end{enumerate}
\end{prop} 

The notion of exceptional sequences is related to Weyl group theory via the following theorem.

\begin{thm} \cite{IS, IT}\label{excseqreflections} 
Let $c$ be a Coxeter element adapted to $Q$ (with respect to sinks, as above). Given a set of $n$ positive roots $\beta_1, \ldots, \beta_n$,  the sequence $(M_{\beta_1}, \ldots, M_{\beta_n})$ of modules associated to the positive roots (by Gabriel's Theorem), is an exceptional sequence if and only if $t_{\beta_1} \ldots t_{\beta_n} = c$.
\end{thm}

We note that the implication from left to right in \ref{excseqreflections} follows from \ref{mutationsexcseqproperties} (5).

We are now able to prove the following theorem.

\begin{thm}\label{+venoncrossing}
There is a bijection between the positive noncrossing partitions and the sincere Hom-free sets in $KQ-\mod$.
\end{thm}

\begin{proof}
A restriction of the map defined in the proof of \cite[Thm. 7.3]{BRT} will give us the required bijection. 

Let the map $\varphi$ from $NC^+ (W_Q)$ to the set of sincere Hom-free sets in $KQ-\mod$ be defined as follows. Given a positive noncrossing partition $u$ with absolute length $r$, there is a $T$-reduced expression for $c$ which has a $T$-reduced expression $t_{\beta_1} \ldots t_{\beta_r}$ for $u$ as a prefix. By \ref{excseqreflections}, the indecomposable modules corresponding to the reflections in this $T$-reduced expression for $c$ give rise to a complete exceptional sequence, and so in particular, $E = (E_1, \ldots, E_r)$, where $E_i$ denotes the indecomposable module associated to $t_{\beta_i}$, is an exceptional sequence. By \ref{mutationsexcseqproperties} (1) we have that  $C(E)$ is equivalent to $KQ^\prime-\mod$ where $Q^\prime$ is a quiver with $r$ vertices and no oriented cycles. We define $\varphi (u)$ to be the set of simple objects $S^\prime := \{S_1^\prime, \ldots S_r^\prime\}$ in $C(E)$. 

Obviously, $S^\prime$ is a Hom-free set in $KQ-\mod$. Suppose, for a contradiction, that $S^\prime$ is not sincere. Observe that the support of $C(E)$, i.e., the support of the modules in $C(E)$, is the same as the support of $E$. Hence, $E$ is not sincere either. But then $u = t_{\beta_1} \ldots t_{\beta_r}$ would lie in the parabolic subgroup generated by the simple roots appearing in the $\beta_i$, for $1\leq i \leq r$, when they are written
as linear combinations of the simple roots. This subgroup is a proper parabolic subgroup since $E$ is not sincere, which contradicts the fact that $u$ is a positive noncrossing partition. Hence $\varphi (u)$ is indeed a sincere Hom-free set in $KQ-\mod$. 

In order to check that this map is well defined we recall some results from \cite{IT}. In this paper the authors give a bijection, called $cox$, between the set of finitely generated wide subcategories of $KQ-\mod$ and $NC (W_Q)$. A wide subcategory is, by definition, an exact abelian subcategory closed under extensions. Any finitely generated wide subcategory $\A$ of $KQ-\mod$ is of the form $\A = KQ^\prime-\mod$, where $Q^\prime$ is a finite acyclic quiver (cf. \cite[Cor. 2.22]{IT}). Given a finitely generated wide subcategory $\A$, $cox (\A)$ is defined to be $t_{S^\prime_1} \ldots t_{S^\prime_k}$, where $(S^\prime_1, \ldots, S^\prime_k)$ are the simple objects of $\A$, ordered into an exceptional sequence. By \cite[Lemma 3.10]{IT}, $cox (\A) = t_{F_1} \ldots t_{F_r}$, for any exceptional sequence $(F_1, \ldots, F_r)$ in $\A$. 

Let $t_{\gamma_1} \ldots t_{\gamma_r}$ be another $T$-reduced expression for $u$, and $E^\prime = (E^\prime_1, \ldots, E^\prime_r)$ be the corresponding exceptional sequence. Note that $C(E)$ and $C(E^\prime)$ are finitely generated wide subcategories of $KQ-\mod$. We have $cox (C(E)) = t_{\beta_1} \ldots t_{\beta_r} = u = t_{\gamma_1} \ldots t_{\gamma_r} = cox (C (E^\prime))$. Because $cox$ is an injective map, we have $C(E) = C(E^\prime)$ and so $\varphi$ is well defined.  

In order to prove that $\varphi$ is an injective map, let $u, v \in NC^+ (W)$ be such that $\varphi (u) = \varphi (v) = S^\prime$. Then, in particular, $u$ and $v$ must have the same absolute length, say $r$. Let $u = t_{\beta_1} \ldots t_{\beta_r}$ and $v= t_{\gamma_1} \ldots t_{\gamma_r}$ be $T$-reduced expressions. Let $E$ and $E^\prime$ be the corresponding exceptional sequences of $u$ and $v$, respectively. We know $S^\prime$ can be ordered into an exceptional sequence in $KQ-\mod$, so let us now view $S^\prime$ as such a sequence rather than just a set of modules. Due to the transitive action of the braid group $B_r$ in $C(E)$, $S^\prime$ can be obtained from $E$ by a sequence of mutations in $C(E)$. Analogously, $E^\prime$ can be obtained from $S^\prime$ by a sequence of mutations in $C(E^\prime)$. Note that all of these mutations can be seen as mutations in $KQ-\mod$. So we have a sequence of mutations in $KQ-\mod$ taking $E^\prime$ to $E$. It follows then by \ref{mutationsexcseqproperties} (5) that $u = t_{\beta_1} \ldots t_{\beta_r} = t_{\gamma_1} \ldots t_{\gamma_r} = v$, as we wanted.

To prove that $\varphi$ is surjective let $\T$ be a sincere Hom-free set in $KQ-\mod$ and $\T \sqcup \U$ be the corresponding Hom-configuration (\ref{sincerity1}). We can choose a refinement $\leq$ of $\preceq$ such that $X \leq Y$ if $X$ has degree $0$ and $Y$ degree $1$. If we order the elements of $\T \sqcup \U$ with respect to this refinement, we obtain an exceptional sequence where the first $k$ terms are the modules, using \ref{Hom-confgs-exc.seq}. Assume $(X_1, \ldots, X_{k}, X_{k+1}, \ldots, X_n)$ is this ordering. Then $(X_1, \ldots, X_{k}, \overline{X}_{k+1}, \ldots, \overline{X}_n)$ is an exceptional sequence in $KQ-\mod$, by \ref{differentnotionexcseq}. By \ref{excseqreflections}, we have
\begin{equation}\label{c}
c = t_{X_1} \ldots t_{X_k} t_{\overline{X}_{k+1}} \ldots t_{\overline{X}_n}.
\end{equation}
Let $u = t_{X_1} \ldots t_{X_k}$. By \eqref{c}, $u$ is a noncrossing partition. Suppose $u$ is not positive. Then $u \in W_{S \setminus \{s\}}$, for some simple reflection $s$. By \ref{Readinglemma}, $t_{X_i} \in W_{S \setminus \{s\}}$, for all $1 \leq i \leq k$. This means that $\T$ doesn't have support at the vertex associated to the simple reflection $s$, which contradicts the hypothesis. Hence $u$ is a positive noncrossing partition.

Since $\T$ is a Hom-free set with $k$ elements, it follows from \ref{Hom-confgsandmodules} that $\T$ is the set of simple objects of $C (\T)$, so $\varphi (u) = \T$, and we are done. 
\end{proof}

\begin{cor}\label{numberHom}
The number of Hom-configurations in $\C(Q)$ is given by the positive Fuss-Catalan number.
\end{cor}

\section{Riedtmann combinatorial configurations}\label{secriedtmannconf}

In this section we give a link between Hom-configurations in $\C(Q)$ and the notion of configurations introduced by Riedtmann.

\begin{definition}\label{defRie}
A set $\T$ of isomorphism classes of indecomposable objects of $\D^b(Q)$ is called a \textit{(Riedtmann) combinatorial configuration} if it satisfies the following properties:
\begin{enumerate}
\item $\Hom (X,Y) = 0$ for all $X, Y$ in $\T$, $X \ne Y$,
\item For all $Z \in \ind \, \D^b(Q)$, there exists $X \in \T$ such that $\Hom (Z, X) \ne 0$.
\end{enumerate}
A combinatorial configuration $\T$ is said to be \textit{$\tau [2]$-periodic} (or just \textit{periodic}) if for every object $X$ in $\T$, we have $\tau^k X [2k] \in \T$ for all $k \in \ZZZ$.
\end{definition}

Riedtmann proved that these configurations are $\tau [2]$-periodic in the cases when $Q$ is of type $A$ or $D$ (cf. \cite{Riedtmann, Riedtmann2}).

We will only consider periodic combinatorial configurations and our aim is to prove that they are in bijection with Hom-configurations in $\C(Q)$.

\begin{lemma}\label{RiedtmannisHomconfg}
If $\T$ is a periodic combinatorial configuration, then the restriction of $\T$ to $\E (Q)$, viewed as a set of objects in $\C(Q)$, is a Hom-configuration in $\C(Q)$. 
\end{lemma}

\begin{proof}
Let us denote the restriction of $\T$ to $\E (Q)$ by $\T^\prime$. It follows from property $1$ of \ref{defRie} and from the periodicity of $\T$ that $\T^\prime$ is a Hom-free set in $\C(Q)$. The maximality follows also from the fact that $\T$ is periodic, since this means that every object $Y$ in $\T$ is of the form $Y = \tau^i Y^\prime [2i]$ for some $i \in \ZZZ$ and some object $Y^\prime \in \T^\prime$, and from property $2$ of \ref{defRie}.
\end{proof}

Using \ref{RiedtmannisHomconfg} and the fact that the number of Hom-configurations in $\C(Q)$ is given by the positive Fuss-Catalan number, it is enough to show that the number of periodic combinatorial configurations is also given by the positive Fuss-Catalan number to get the bijection between these two notions of configurations.

In order to check this, we use some results presented in \cite{BLR}. Namely the authors introduce another notion of configuration, which we shall refer to as \textit{BLR-configuration}. Such configurations are periodic (cf. \cite[Prop.1.1]{BLR}), and they are a subset of the set of periodic combinatorial configurations (for more details see the introduction in \cite{BLR}). 

\begin{thm}\label{theoremconfigurations}
There is a bijection between the following objects:
\begin{enumerate}
\item BLR-configurations in $\D^b(Q)$;
\item periodic combinatorial configurations in $\D^b(Q)$;
\item Hom-configurations in $\C(Q)$.
\end{enumerate}
\end{thm}

\begin{remark}
We have just seen that the following hold: 
\begin{align*}
\{ \text{BLR-configurations} \} &\subseteq \{ \text{periodic combinatorial configurations} \}\\
                                &\hookrightarrow \{ \text{Hom-configurations in } \C(Q) \}.
\end{align*}
\end{remark}

Hence, the only thing we need to check to prove \ref{theoremconfigurations} is that the number of BLR-configurations is given by the positive Fuss-Catalan number, i.e., the number of Hom-configurations in $\C(Q)$, by \ref{numberHom}.\\

If $Q$ is a quiver of type $A$, Bretscher, L\"aser and Riedtmann prove that BLR-configurations in $\D^b(Q)$ are in bijection with pedigrees with $n$ vertices (cf. main theorem in introduction and section 6.2 in \cite{BLR}). By definition, a \textit{pedigree} is a subtree of the oriented tree:
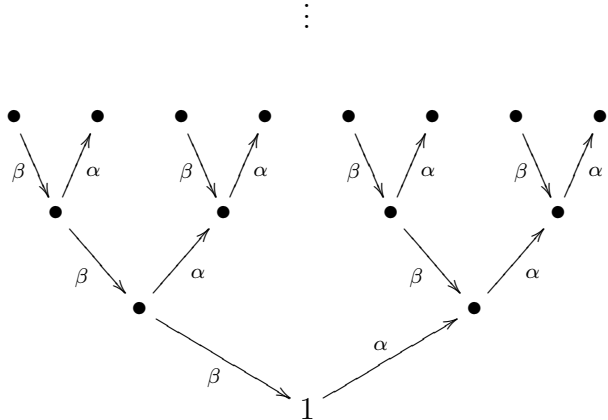
\[
\xymatrix@C=0.1cm{&&&&&&& \vdots \\
\bullet \ar[dr]_{\beta} & & \bullet & & \bullet \ar[dr]_{\beta} & & \bullet & & \bullet \ar[dr]_{\beta} & & \bullet & & \bullet \ar[dr]_{\beta} & & \bullet \\
& \bullet \ar[drr]_{\beta} \ar[ur]_{\alpha} &&&& \bullet \ar[ur]_{\alpha} &&&& \bullet \ar[drr]_{\beta} \ar[ur]_{\alpha} &&&& \bullet \ar[ur]_{\alpha} \\
&&& \bullet \ar[urr]_{\alpha} \ar[drrrr]_{\beta} &&&&&&&& \bullet \ar[urr]_{\alpha} \\
&&&&&&& 1 \ar[urrrr]^{\alpha}} 
\]

which contains the lowest vertex $1$.

Pedigrees with $n$ vertices are in 1-1 correspondence with binary trees. We recall that a binary tree is a rooted tree (trees are drawn growing downwards, by convention) in which each vertex $i$ has at most two children, i.e., vertices adjacent to $i$ which are below it in the tree. Each child of a vertex is designated as its left or right child.

The correspondence is described as follows: the lowest vertex $1$ corresponds to the root, and $x$ is a right (left) child of $y$ if and only if we have $\xymatrix{y \ar[r]^\alpha & x}$ ($\xymatrix{x \ar[r]^\beta & y}$) in the pedigree.

It is known that the number of binary trees with $n$ vertices is given by $\frac{1}{n+1} \binom{2n}{n}$, which is the positive Fuss-Catalan number $C^+ (A_n)$. So we are done in type $A$.\\   

It follows from \cite[Prop.7.2]{BLR} that the number of BLR-configurations in type $D_n$ is also given by the positive Fuss-Catalan number (see end of section 7.5 in \cite{BLR}). So it remains to check type $E$.

In \cite{BLR}, the authors define two classes of BLR-configurations, the isomorphism classes and the equivalence classes. Isomorphisms of BLR-configurations come from automorphisms of the translation quiver, which are given by $\tau^k$, with $k \in \ZZZ$, or by reflection in a horizontal line in type $E_6$. Two BLR-configurations are said to be \textit{equivalent} if they are isomorphic or one is isomorphic to the reflection of the other at a vertical line. 

In types $E_7$ and $E_8$ there is no reflection along a horizontal line. Hence, each isomorphism class has $h-1$ elements: a representative $\T$ and $\tau^k (\T)$, with $1 \leq k \leq h-2$, as $\tau^{h-1} (\T) = \T$. Hence the number of BLR-configurations is given by multiplying the number of isomorphism classes with $h-1$. In \cite{BLR} the authors state that there are $143$ and $598$ isomorphism classes for type $E_7$ and $E_8$ respectively. Since $h-1$ equals $17$ for type $E_7$ and $29$ for type $E_8$, the number of BLR-configurations is $2431$ and $17342$ for type $E_7$ and $E_8$ respectively, which is the positive Fuss-Catalan number, as we wanted.

For type $E_6$ there are $17$ equivalence classes. The authors of \cite{BLR} list a representative for each of these equivalence classes. One can see that $12$ of these equivalence classes are invariant under the vertical reflection. Thus there are $12 + 2 \times 5 = 22$ isomorphism classes. One can easily check that $6$ of these isomorphism classes are invariant under the horizontal reflection. %($4$ are invariant under the vertical reflection, and the other two come from an equivalence class which is not invariant under the vertical reflection and the isomorphism class of its vertical reflection)
Therefore, there are $6 + 16 \times 2 = 38$ BLR-configurations up to $\tau$-translation. Hence there are $38 \times (h-1) = 38 \times 11 = 418$ BLR-configurations in total, which is the positive Fuss-Catalan number for type $E_6$.

\section{Riedtmann's bijection for type $A$}\label{secriedtmann}

The notion of classical noncrossing partitions of $\{1, \ldots, n\}=[n]$ was introduced by Kreweras \cite{Kr} in 1972 and it is defined as follows.

\begin{definition} \cite{Kr}
A classical noncrossing partition of $[n]$ is a partition $\P = \{\B_1, \ldots, \B_m\}$ of the set $\{1, \ldots, n\}$, where we call $\B_i$ a block of $\P$ for $1 \leq i \leq m$, with the property that if $1 \leq a < b < c < d \leq n$, with $a, c \in \B_i$ and $b, d \in \B_j$, then $\B_i = \B_j$.  
\end{definition}

One can interpret this as being a partition of the vertices of a regular $n$-gon, whose vertices are ordered clockwise from $1$ to $n$, such that the convex hulls of its blocks are disjoint from each other. 

The set of classical noncrossing partitions of $[n]$ form a poset under refinement of partitions, and we denote this poset by $NC (n)$. It was proved by Biane that $NC (n)$ and $NC (A_{n-1})$ are isomorphic posets:

\begin{thm} \cite[Thm 1]{Biane}
Given a permutation $\pi$ of $[n]$ write it as a product of disjoint cycles (including $1$-cycles) and let $\{\pi\}$ denote the partition of $[n]$ given by these cycles. The map $\pi \mapsto \{\pi\}$ is a poset isomorphism between $NC (A_{n-1})$ and $NC (n)$.
\end{thm}

In this section $Q$ denotes the quiver of type $A_n$ with linear orientation:
\[
Q: \xymatrix{n \ar@{->}[r] & n-1 \ar@{->}[r] & \ldots \ar@{->}[r] & 1}.
\]

Riedtmann \cite{Riedtmann} proved that there is a bijection between the set of combinatorial configurations of $\D^b(Q)$ and $NC (n)$. In order to describe this map we need the following notation.

We know there is a bijection between the AR-quiver $\Gamma (\D^b(Q))$ of $\D^b (Q)$ and the stable translation quiver $\ZZZ Q^{op}$, which is defined as follows:
\begin{enumerate}
\item Vertices: $(\ZZZ Q^{op})_0 := \ZZZ \times Q_0^{op}$,
\item Arrows: for vertices $(x,a), (y,b)$ of $\ZZZ Q^{op}$, $\xymatrix{(x,a) \ar@{->}[r] & (y,b)}$ is an arrow in $\ZZZ Q^{op}$ if $x = y$ and $\xymatrix{ a \ar@{->}[r] & b}$ is an arrow in $Q^{op}$ or $y = x+1$ and $\xymatrix{ b \ar@{->}[r] & a}$ is an arrow in $Q^{op}$.
\end{enumerate}

This bijection can be chosen so that the indecomposable projective $P_i$ corresponds to $(1,i)$, for $i \in [n]$. Observe that the indecomposable $KQ$-modules are the objects of $\ZZZ Q^{op}$ written in the form $(i,j)$ with $2 \leq i+j \leq n+1$, with $i \geq 1$. 

Recall that combinatorial configurations of type $A$ are periodic and so by \ref{RiedtmannisHomconfg} they can be regarded as Hom-configurations in $\C(Q)$. Moreover, it was seen in Section $6$ that the map in \ref{RiedtmannisHomconfg}, which is the restriction of a combinatorial configuration to the fundamental domain $\E(Q)$, is in fact a bijection. The composition of this bijection with Riedtmann's map (cf. \cite[2.6]{Riedtmann}) can be described as follows.

\begin{thm} \cite[2.6]{Riedtmann} \label{Riedtbijection}
Let $\P = \{\B_1, \ldots, \B_m\}$ be a classical noncrossing partition of the vertices of a regular $n$-gon, and assume the elements of each $\B_i$ are in numerical order. Given $k \in [n]$, let $\B = \{k_1, \ldots, k_s \}$ be the block that contains $k$. So $k = k_r$ for some $1 \leq r \leq s$. Let $\psi (k_r): = (k_{(r+1) \, \mod \, s} - k_r) \, \mod \, n$. Here we use modular arithmetic using the representatives $\{1, 2, \ldots, l\}$ when working $\mod \, l$. 

Then the set $\{(i, \psi (i)) \mid i \in [n] \}$ is a Hom-configuration in $\C(Q)$ and the map defined this way, which we will call $\gamma$, is a bijection between $NC (n)$ and the set of Hom-configurations in $\C(Q)$.
\end{thm}

\begin{ex}
Consider the noncrossing partition $\P = \{ \{1,3\}, \{2\}, \{4\}\}$ of $\{1, 2, 3, 4\}$. Then the image under $\gamma$ is $\{(1,2),(2,4), (3,2), (4,4)\} = \{ 12, 1[1], 34, 3[1]\}$, with the notation we introduced in \ref{examplenotation} (but note that the quiver we are using here has linear orientation). 
\end{ex}

Our aim is to check that the composition of the bijections in \ref{+venoncrossing} and \ref{1-1corsincere} gives a generalization of this result. 

First, we will give a combinatorial description for the positive noncrossing partitions of type $A$.

\begin{prop}\label{combinatorial+ve}
A classical noncrossing partition of $[n+1]$ is positive if and only if the vertices $1$ and $n+1$ lie in the same block.
\end{prop}

\begin{proof}
Let $\P = \{\B_1, \ldots, \B_m\}$ be a classical noncrossing partition of $[n+1]$ and $u$ be the corresponding noncrossing partition of type $A_n$. 

Suppose $1$ and $n+1$ don't lie in the same block. Let $\B_1$ be the block which contains the vertex $1$, and write $\B_1 = \{1, k_2, \ldots, k_s\}$ in numerical order. By assumption, $k_s \ne n+1$. This block corresponds to the cycle $(1 \, k_2 \, \ldots \, k_s)$, which can be written as a product of reflections (i.e., transpositions in this case):
\[
(1 \, k_2 \, \ldots \, k_s) = (1 \, k_2) (k_2 \, k_3) \ldots (k_{s-1} \, k_s) = t_{\alpha_1 + \ldots + \alpha_{k_2 -1}} t_{\alpha_{k_2} + \ldots + \alpha_{k_3 -1}} \ldots t_{\alpha_{k_{s-1}} + \ldots + \alpha_{k_s-1}}. 
\]
This element belongs to the parabolic subgroup $W_1$ generated by the simple reflections $s_{\alpha_1}, s_{\alpha_2}, \ldots, s_{\alpha_{k_s}-1}$. Note that this subgroup is proper as $s_{\alpha_{k_s}} \not\in W_1$.

Since $\P$ is noncrossing, there are no vertices $l$ and $m$ lying in the same block with $l < k_s < m$. Hence, the cycle $c_i$ corresponding to the block $\B_i$ lies in a parabolic subgroup $W_i$ which does not contain $s_{\alpha_{k_s}}$, and so $u = \prod_{i=1}^m c_i$ belongs to the product of the parabolic subgroups $W_i$ ($1 \leq i \leq m$), which is a proper parabolic subgroup. Hence the noncrossing partition $u$ is not positive.

Now suppose $1$ and $n+1$ are in the same block, say $\B_1$. We have $$\B_1 = \{1, k_{1 , 2}, \ldots, k_{1 , s-1}, n+1\}$$ with $1 < k_{1 , 2} < \ldots < k_{1 , s-1} < n+1$. The corresponding cycle $c_1 = (1 \, k_{1 , 2} \ldots \, k_{1 , s-1} \, n+1)$ can be written in the form:
\[
c_1 = t_{\alpha_1 + \ldots + \alpha_{k_{1 , 2}-1}} t_{\alpha_{k_{1 , 2}} + \ldots + \alpha_{k_{1 , 3}-1}} \cdots t_{\alpha_{k_{1 , s-1}} + \ldots + \alpha_n}.
\]

Note that $l_T (c_1) = s-1$, since $s$ is the length of the cycle (cf. \cite[Prop. 2.3]{Brady}). Hence the product of reflections above is a $T$-reduced expression for $c_1$. Moreover, if $c_i$ is the cycle corresponding to $\B_i$, we have an expression for $u$ as a product of disjoint cycles, $u = \prod_{i=1}^m c_i$, and so $l_T (u) = \sum_{i=1}^m l_T(c_i)$ (cf. \cite[Lemma 2.2]{Brady}).

Let $E$ be the exceptional sequence in $KQ-\mod$ associated to this $T$-reduced expression of $u$, by \ref{excseqreflections}. Due to the $T$-reduced expression for $c_1$, we have that $E$ has support on every vertex of $Q$, i.e., $E$ is sincere. 

Consider $C(E) \simeq KQ^\prime-\mod$, where the number of vertices of $Q^\prime$ equals $r$, the number of terms in $E$ (cf. \cite[Lemma 5]{Bill}). Let $S^\prime = \{S^\prime_1, \ldots, S^\prime_r \}$ be the set of simple objects in $C(E)$, ordered into an exceptional sequence. Due to the transitive action of the set of mutations on complete exceptional sequences in $C(E)$, $E$ can be obtained from $S^\prime$ by a sequence of mutations. Hence $u = t_{\underline{dim} \, S^\prime_1} \ldots t_{\underline{dim} \, S^\prime_r}$, by \ref{mutationexcseqreflection}. 

On the other hand, $supp \, S^\prime = supp \, C(E) = supp \, (E)$, which implies that $S^\prime$ is a sincere Hom-free set in $KQ-\mod$. By \ref{+venoncrossing}, $\varphi^{-1} (S^\prime) = t_{\underline{dim} \, S^\prime_1} \ldots t_{\underline{dim} \, S^\prime_r}$ is a positive noncrossing partition, i.e., $u \in NC^+ (A_n)$, as we wanted.
\end{proof}

The following proposition follows easily from \ref{combinatorial+ve}.

\begin{prop}
Given a classical noncrossing partition $\P = \{\B_1, \ldots, \B_m\}$ of $[n]$, where $1 \in \B_1$, let $f (\P)$ be the partition $\{\B_1^\prime, \ldots, \B_m^\prime\}$ of $[n+1]$ defined as follows:
\begin{equation*}
\B_i^\prime = 
 \begin{cases}
 \B_1 \cup \{n+1\} & \text{ if } i = 1 \\
 \B_i & \text{ if } i \ne 1.
 \end{cases}
\end{equation*}
Then $f(\P)$ is a positive noncrossing partition and $f: NC (n) \rightarrow NC^+ (n+1)$, where $NC^+ (n+1)$ is the image of $NC^+ (A_n)$ under the isomorphism between $NC (A_n)$ and $NC (n+1)$, is a bijection.  
\end{prop}

\begin{thm}
Let $\rho: NC^+ (A_n) \rightarrow \{ \text{Hom-configurations in } \C(Q)\}$ be the composition of the bijection $\varphi$ in \ref{+venoncrossing} followed by the bijection $\beta$ in \ref{1-1corsincere}. Let $\gamma$ be Riedtmann's bijection (see \ref{Riedtbijection}). Then we have that $\rho^{-1} \circ \gamma =f$.
\end{thm}

\begin{proof}
We recall that here we are using the notation for the stable translation quiver $\ZZZ Q^{op}$. Observe that the element $(i, j) \in \ZZZ Q^{op}$ with $i \geq 1$ and $i+j \leq n+1$ corresponds to the indecomposable module $M_{ij}$ whose dimension vector is given by 
\begin{equation*}
(\underline{dim} \, M_{ij})_l =
 \begin{cases}
 1 & \text{ if } l \in \{ i, i+1, \ldots, i+j-1\} \\
 0 & \text{ otherwise. }
 \end{cases}
\end{equation*}
Note also that this indecomposable module corresponds to the transposition $(i \, i+j) = t_{\alpha_i + \ldots + \alpha_{i+j-1}}$. 

Let $\P = \{\B_1, \ldots, \B_m\} \in NC (n)$. Let $\leq$ be a refinement to a total order of the partial order $\preceq$ such that the elements of the Hom-configuration $\gamma (\P)$ ordered with respect to this refinement form  an exceptional sequence where the modules are the first elements. We can assume this refinement satisfies the following property: if the indecomposable objects $(i, \psi (i))$ and $(j, \psi (j))$ have the same degree, then $(i, \psi (i)) \leq (j, \psi (j))$ if $i < j$. This means that the modules in $\gamma (\P)$ are ordered into an exceptional sequence in the following way:
\[
M_{i_1, \psi (i_1)}, \ldots, M_{i_k, \psi (i_k)},
\]
where $1 \leq i_1 < \ldots < i_k$ and $i_j + \psi (i_j) \leq n+1$. So we have 
\begin{equation}\label{image}
\rho^{-1} (\gamma (\P)) = t_{\underline{dim} M_{i_1,\psi (i_1)}} \cdots t_{\underline{dim} M_{i_k, \psi (i_k)}}.
\end{equation}

Note that if $(i, \psi (i)), (j, \psi (j)) \in \gamma (\P) \cap KQ-\mod$ are such that $i$ and $j$ belong to different blocks, then the corresponding reflections commute.

Hence we can write 
\begin{equation}\label{varphi}
\rho^{-1} (\gamma (\P )) = \prod_{j= 1}^m \prod_{\substack{i \in \B_j \\ i + \psi (i) \leq n+1}} t_{\underline{dim} \, M_{i\psi (i)}},
\end{equation}
where the product corresponding to each block $\B_j$ respects the order in \eqref{image}.

Given this, consider the block $\B_1 = \{k_{11} = 1, k_{12}, \ldots, k_{1,r_1-1}, k_{1r_1} \}$ of $\P$. This block gives rise to the following elements in $\gamma (\P)$:
\[
(1, k_{12}-1), (k_{12}, k_{13}-k_{12}), \ldots , (k_{1,r_1-1}, k_{1r_1} - k_{1,r_1-1}), (k_{1r_1}, n-k_{1r_1} +1).
\]
We denote this set of the elements by $\T_1$.

The corresponding indecomposable objects lie in $KQ-\mod$, since they are of the form $(i,j)$ with $i \geq 1$ and $i + j \leq n+1$.

The reflections associated to the elements of $\T_1$ are:
\[
(1 \, k_{12}), (k_{12} \, k_{13}), \ldots , (k_{1r_1} \, n+1), 
\]
respectively.

The part of the product in \eqref{varphi} corresponding to $\B_1$ is the following product: 
\begin{align*}
\prod_{i \in \B_1} t_{\underline{dim} \, M_{i, \psi (i)}} &= (1 \, k_{12}) (k_{12} \, k_{13}) \ldots (k_{1r_1} \, n+1) \\
                                                          &= (1 \, k_{12} \, k_{13} \, \ldots k_{1r_1} \, n+1).
\end{align*}                                                          

Let $\B_j$ be any other block of $\P$, and write $\B_j = \{k_{j1}, k_{j2}, \ldots, k_{jr_j} \}$, with $k_{j1} < k_{j2} < \ldots < k_{jr_j}$. Following the same argument as before, this block gives rise to the following objects in $\gamma (\P)$:
\[
(k_{j1}, k_{j2} -k_{j1}), (k_{j2}, k_{j3}-k_{j2}), \ldots, (k_{j,r_j-1}, k_{jr_j}-k_{j,r_j-1}), (k_{jr_j}, n-k_{jr_j}+k_{j1}).
\]

All these objects but the last one lie in $KQ-\mod$ (note that $(k_{jr_j}, n-k_{jr_j} + k_{j1}) \not\in KQ-\mod$ since $k_{jr_j} + (n-k_{jr_j}+k_{j1}) = n + k_{j1} \geq n+2$ as $k_{j1} \ne 1$).

We have:
\begin{align*}
\prod_{\substack{i \in \B_j \\ i + \psi (i) \leq n+1}} t_{\underline{dim} \, M_{i,\psi (i)}} & = (k_{j1} \, k_{j2}) (k_{j2} \, k_{j3}) \ldots (k_{j,r_j-1} \, k_{jr_j}) \\
& = (k_{j1} \, k_{j2} \, \ldots k_{j,r_j-1} \, k_{jr_j}).
\end{align*}

Hence the blocks of $\rho^{-1} (\gamma (\P))$ are $\{k_{j1}, \ldots, k_{jr_j}\} = \B_j$, with $2 \leq j \leq m$, and $\{1, k_{12}, \ldots, k_{1r_1}, n+1\} = \B_1 \cup \{n+1\}$, which allow us to conclude that $\rho^{-1} \circ \gamma = f$, as we wanted.
\end{proof}

\section{Mutations of Hom-configurations}\label{secmutations}

In this section we give a definition of mutation of Hom-configurations in $\C(Q)$, which will rely on \ref{mainprop}. The first thing we need to do is to generalize, in the obvious way, this result. 

\begin{cor}\label{genmainres}
Let $\{X_1, \ldots, X_k \}$ be a Hom-free set  in $\ind \, \C (Q)$. Then 
\[
\,^\perp \{X_1, \ldots, X_k \}^\perp = \{Y \in \C (Q) \, \mid \, \Hom_{\C (Q)} (X_i, Y) = 0 = \Hom_{\C (Q)} (Y, X_i), \forall i \in [k] \}
\]
is equivalent to $\C (Q^\prime)$ where $Q^\prime$ is a disjoint union of quivers of Dynkin type and whose number of vertices is $n-k$.
\end{cor}
\begin{proof}
We prove this by induction on $k$. The case when $k = 1$ is \ref{mainprop}. Let $\{X_1, \ldots, X_k, X_{k+1}\}$ be a Hom-free set in $\C (Q)$. Then so is $\{X_1, \ldots, X_k \}$ and by induction we have
\begin{equation}\label{induction}
\,^\perp \{X_1, \ldots, X_k\}^\perp \simeq \C (Q^\prime),
\end{equation}
where $Q^\prime$ is a disjoint union of quivers of Dynkin type, whose sum of vertices is $n - k$, in other words, $Q^\prime = \sqcup_{i=1}^t Q^i$, where each $Q^i$ is a Dynkin quiver and $\sum_{i=1}^t |Q^i_0| = n-k$.

We have $X_{k+1} \in \,^\perp \{X_1, \ldots, X_k \}^\perp_Q$, and so by \eqref{induction}, there exists a unique $i \in [t]$ for which $X_{k+1} \in \C (Q^i)$. Given $Y \in \C (Q^j)$, with $j \ne i$, we have
\[
\Hom_{\C (Q)} (X_{k+1}, Y) = 0 = \Hom_{\C (Q)} (Y, X_{k+1})
\]
since $\C (Q^l)$, $l \in [t]$, are pairwise orthogonal. Hence $\sqcup_{j \ne i} \C (Q^j) \subseteq \,^\perp (X_{k+1})^\perp_Q$.
Therefore,
\begin{equation*}
\begin{split}
\,^\perp \{X_1, \ldots, X_{k+1} \}^\perp_Q & = \,^\perp (X_{k+1})^\perp_Q \cap \,^\perp \{X_1, \ldots, X_k \}^\perp_Q \\
                                           & \simeq \,^\perp (X_{k+1})^\perp_Q \cap (\sqcup_{l \in [t]} \C (Q^l)) \, (by \eqref{induction}) \\
                                           & = ( \,^\perp (X_{k+1})^\perp_Q \cap \C (Q^i)) \sqcup ( \,^\perp (X_{k+1})^\perp_Q \cap \sqcup_{j \ne i} \C (Q^j)) \\
                                           & = ( \,^\perp (X_{k+1})^\perp_Q \cap \C (Q^i))) \sqcup (\sqcup_{j \ne i} \C (Q^j)).
\end{split}
\end{equation*}

Note that $\,^\perp (X_{k+1})^\perp_Q \cap \C (Q^i)) \simeq \,^\perp (X_{k+1})^\perp_{Q^i}$, which, by \ref{mainprop}, is equivalent to $\C (Q^{\prime \prime})$, where $Q^{\prime \prime}$ is a disjoint union of quivers of Dynkin type whose sum of vertices is $|Q^i_0| - 1$, and we are done.
\end{proof}

In representation theory, mutations are operations that act on a certain class of objects, and construct a new one from a given one by replacing a summand. For example, mutations of exceptional sequences and cluster mutations are of this type. In this case however, this won't make much sense. Indeed, let $\T = \{X_1, \ldots, X_n\}$ be a Hom-configuration in $\C (Q)$, and suppose we remove one object $X_i$. Then, by \ref{genmainres}, $\,^\perp (\T \setminus X_i)^\perp_Q$ is equivalent to $\C (Q^\prime)$ where $Q^\prime$ is of type $A_1$. Hence, the only completion of $\T$, i.e., the only object $Y$ of $\C (Q)$ for which $\T \cup Y$ is a Hom-configuration is $X_i$. Therefore, in order to define mutations of Hom-configurations, we need to remove more than one object.

A particular case of \ref{genmainres}, which will be useful later, is the following:

\begin{cor}\label{Cor3}
Let $\{X_1, \ldots, X_n\}$ be a Hom-configuration in $\C (Q)$. If we remove two objects, say $X_i$, and $X_j$, with $i \ne j$, then $\,^\perp \{X_k \, \mid \, k \ne i, j\}^\perp_Q \simeq \C (Q^\prime)$, where $Q^\prime$ is a quiver with two vertices which are either connected by a single arrow, i.e., $Q^\prime$ is of type $A_2$, or they are disconnected, i.e., $Q^\prime$ is of type $A_1 \times A_1$.
\end{cor}

\begin{lemma}\label{Hom-configsQ_2}
Let $Q^\prime$ be a quiver of type $A_2$. Then the only two Hom-configurations in $\C (Q^\prime)$ are the set of simple modules, and $\{P, P^\prime [1]\}$, where $P$ is the injective-projective $KQ^\prime$-module and $P^\prime$ is the simple projective $KQ^\prime$-module.
\end{lemma}

Given these results, we are able to give a definition of mutation of Hom-configurations in $\C (Q)$. 

\begin{definition}\label{mutation}
Let $i, j \in [n], i < j$. The \textit{mutation $\mu_{i,j}$ of the Hom-configuration} $\T := \{X_1, \ldots, X_n\}$ in $\C (Q)$ is defined as follows:
\begin{enumerate} 
\item  If $\,^\perp \{\T \setminus \{X_i,X_j\}\}^\perp_Q$ is equivalent to $\C (Q^\prime)$ where $Q^\prime$ is a quiver of type $A_2$, then $\mu_{i,j} (\T)$ is the Hom-configuration obtained from $\T$ by replacing the Hom-configuration $\{X_i, X_j\}$ in $\C (Q^\prime)$ by the other possible Hom-configuration $\{X_i^\prime, X_j^\prime\}$ in this category (cf. \ref{Hom-configsQ_2}). 
\item For the remaining case, i.e., if $\,^\perp \{\T \setminus \{X_i,X_j\}\}^\perp_Q$ is equivalent to $\C (Q^\prime)$ where $Q^\prime$ is of type $A_1 \times A_1$, then $\mu_{i,j} (\T) = \T$.
\end{enumerate}
\end{definition}  

\begin{ex}
Consider the quiver $Q$ of type $A_3$: $\xymatrix{3 \ar@{->}[r] & 2 \ar@{->}[r] & 1}$. The following figure shows the graph of mutations of the Hom-configurations in $\C(Q)$, where the vertices correspond to Hom-configurations and the edges correspond to mutations:
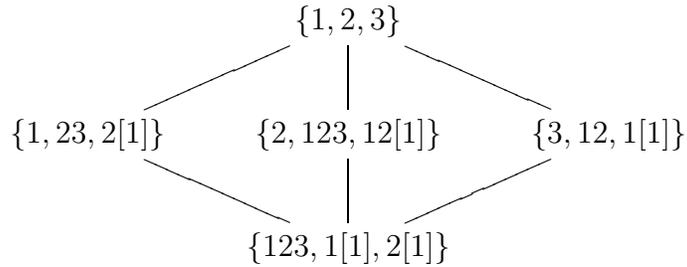
\[
\xymatrix{ & \{1, 2, 3 \} \ar@{-}[d] \ar@{-}[dr] \\
\{1, 23, 2[1]\} \ar@{-}[ur] \ar@{-}[dr] & \{2, 123, 12[1]\} \ar@{-}[d] & \{3, 12, 1[1]\} \\
& \{123, 1[1], 2[1]\} \ar@{-}[ur]}
\]
\end{ex}

\begin{prop}
The graph of mutations $\G (Q)$ of Hom-configurations in $\C (Q)$ is connected.
\end{prop}

\begin{proof}
We prove this by induction on the number of vertices $n$. For $n = 1$ there is nothing to prove. 

It is easy to check that this proposition holds in the cases when $Q$ is of type $A_2$, $A_3$ and $D_4$.

Let now $Q$ be any other Dynkin quiver with $n$ vertices. Note that given a vertex $i$ of $Q$ there exists a vertex $j$ which is not a neighbor of $i$, i.e., there is no arrow between $i$ and $j$. First we claim that two Hom-configurations $\T$ and $\T^\prime$ in $\C(Q)$ with a common object $X$ are connected by a path in $\G (Q)$. Let $F$ be the equivalence between $\,^\perp X^\perp$ and $\C(Q^\prime)$, where $Q^\prime$ is a Dynkin quiver with $n-1$ vertices (see \ref{mainprop}). We have that $F (\T \setminus X)$ and $F(\T^\prime \setminus X)$ are Hom-configurations in $\C(Q^\prime)$. By induction $\G (Q^\prime)$ is connected, so $F (\T \setminus X)$ and $F(\T^\prime \setminus X)$ are connected by a path in $\G(Q^\prime)$, i.e., there is a sequence of mutations $\mu_1, \ldots, \mu_k$ such that $\mu_1 \ldots \mu_k (F(\T \setminus X)) = F(\T^\prime \setminus X)$. This sequence of mutations can be lifted to a sequence of mutations in $\C(Q)$ fixing the object $X$. Hence $\T$ and $\T^\prime$ are connected by a path in $\G(Q)$, as we wanted to prove.

Now fix a Hom-configuration $\T$ in $\C(Q)$. We want to prove that $\T$ is connected to the simple Hom-configuration $\S$ by a path in $\G(Q)$. By \ref{Hom-confgsandmodules} (2) there exists a $KQ$-module $X$ in $\T$, and there is a sequence of reflection functors $R = R_{i_k} \ldots R_{i_1}$, where $i_j$ is a sink in $\sigma_{i_{j-1}} \ldots \sigma_{i_1} (Q)$, such that $R_{i_k} \ldots R_{i_1} (X)$ is a simple $K \sigma_R (Q)$-module. Since $R(\T)$ is a Hom-configuration in $\C(\sigma_R (Q))$ and it contains a simple module, it follows from what we claimed above that $R(\T)$ and the set of simple $K\sigma_R (Q)$-modules $\S_{\sigma_R (Q)}$ are connected by a path in $\G(\sigma_R (Q))$. Note that $i_k$ is a source in $\sigma_R (Q)$ and by assumption, there is a vertex $j$ which is not a neighbor of $i_k$. Hence $R_{i_k}^{-} (S_{\sigma_R (Q)} (j)) = S_{\sigma_{i_k} \sigma_R (Q)} (j)$ (cf. \cite[VII.5.4]{ASS}). Therefore $R_{i_k}^- (\S_{\sigma_R(Q)})$ is connected to $\S_{\sigma_{i_k} \sigma_R(Q)}$ by a path in $\G(\sigma_{i_k} \sigma_R (Q))$ since it contains a simple $K \sigma_{i_k} \sigma_R(Q)$-module. Since $R_{i_k}^- R(\T)$ and $R_{i_k}^- (\S_{\sigma_R(Q)})$ are connected by a path in $\G(\sigma_{i_k} \sigma_R (Q))$, it follows that $R_{i_k}^-R(\T)$ is also connected to $\S_{\sigma_{i_k}\sigma_R(Q)}$ by a path in $\G(\sigma_{i_k} \sigma_R (Q))$. Using the same argument when we apply the reflections $R_{i_{k-1}}^{-} \ldots R_{i_1}^{-}$, we deduce that $\T$ is connected to $\S_Q$ by a path in $\G(Q)$.
\end{proof}

\subsection*{Acknowledgments} The author would like to express her gratitude to her supervisor, Robert Marsh, for his help and advice. The author would also like to thank Aslak Buan, Idun Reiten and Hugh Thomas for making available a preliminary version of their paper \cite{BRT}, and to Aslak Buan for helpful conversations. She would also like to thank Fundac\~ao para a Ci\^encia e Tecnologia, for their financial support through Grant SFRH/ BD/ 35812/ 2007.

%%%%%%%%%%%%%%%%End of article
%%%%%%%%%%%%%%%%Bibliography

\end{document}